\documentclass[letter,12pt]{amsart}

\usepackage{tikz}
\usetikzlibrary{decorations.markings}
\usetikzlibrary{shapes}
\usetikzlibrary{backgrounds}
\usepackage{subfig}
\usepackage{amssymb} 
\usepackage[all]{xy}

\usepackage{hyperref}
\def\thetitle{{Embedability between right-angled Artin groups}}
\hypersetup{
    pdftitle=   \thetitle,
   pdfauthor=  {Sang-hyun Kim and Thomas Koberda}
}
\usepackage{enumerate}
\makeatletter
\let\@@enum@org\@@enum@
\def\@@enum@[#1]{\@@enum@org[\normalfont #1]}
\makeatother

\newtheorem{thm}{Theorem}
\newtheorem{lem}[thm]{Lemma}
\newtheorem{cor}[thm]{Corollary}
\newtheorem{prop}[thm]{Proposition}
\newtheorem{con}[thm]{Conjecture}
\newtheorem{que}{Question}

\newtheorem*{thm6}{Theorem 6}

\theoremstyle{remark}

\newtheorem*{rem}{Remark}

\theoremstyle{definition}
\newtheorem{defn}[thm]{Definition}

\newcommand\form[1]{\langle #1\rangle}
\def\opp{^\mathrm{opp}}

\newcommand\N{\mathbb{N}}

\newcommand\Z{\mathbb{Z}}

\newcommand\supp{\operatorname{supp}}

\newcommand\lk{\operatorname{Lk}}
\newcommand\st{\operatorname{St}}
\newcommand\diam{\operatorname{diam}}

\newcommand\gex{{\Gamma}^e}
\newcommand\aga{{A(\Gamma)}}

\newcommand\gam{\Gamma}
\newcommand\bZ{\mathbb{Z}}
\newcommand\Mod{Mod}
\newcommand\Homeo{Homeo}
\newcommand\mC{\mathcal{C}}

\begin{document}

\title\thetitle

\author{Sang-hyun Kim}
\address{Department of Mathematical Sciences, KAIST, 335 Gwahangno, Yuseong-gu, Daejeon 305-701, Republic of Korea}
\email{shkim@kaist.edu}
\thanks{The first named author is supported by the Basic Science Research Program (2010-0023515) and the Mid-Career Researcher Program (2010-0027001) through the National Research Foundation funded by the Ministry of Education, Science and Technology of Korea.}

\author{Thomas Koberda}
\address{Department of Mathematics, Harvard Univerity, 1 Oxford St., Cambridge, MA 02138, USA}
\email{koberda@math.harvard.edu}
\thanks{}
\date{\today}
\begin{abstract}
In this article we study the right-angled Artin subgroups of a given right-angled Artin group.  Starting with a graph $\gam$, we produce a new graph through a purely combinatorial procedure, and call it the extension graph $\gex$ of $\gam$.  We produce a second graph $\gex_k$, the clique graph of $\gex$, by adding extra vertices for each complete subgraph of $\gex$.  We prove that each finite induced subgraph $\Lambda$ of $\gex$ gives rise to an inclusion $A(\Lambda)\to A(\gam)$.  Conversely, we show that if there is an inclusion $A(\Lambda)\to A(\gam)$ then $\Lambda$ is an induced subgraph of $\gex_k$.  These results have a number of corollaries.  Let $P_4$ denote the path on four vertices and let $C_n$ denote the cycle of length $n$.  We prove that $A(P_4)$ embeds in $A(\gam)$ if and only if $P_4$ is an induced subgraph of $\gam$.  We prove that if $F$ is any finite forest then $A(F)$ embeds in $A(P_4)$.  We recover the first author's result on co--contraction of graphs, and prove that if $\gam$ has no triangles and $A(\gam)$ contains a copy of $A(C_n)$ for some $n\geq 5$, then $\gam$ contains a copy of $C_m$ for some $5\le m\le n$.  We also recover Kambites' Theorem, which asserts that if $A(C_4)$ embeds in $A(\gam)$ then $\gam$ contains an induced square.  Finally, we determine precisely when there is an inclusion $A(C_m)\to A(C_n)$ and show that there is no ``universal" two--dimensional right-angled Artin group.
\end{abstract}
\maketitle

\section{Introduction}
This article gives a systematic study of the existence of embeddings between right-angled Artin groups.

By a \emph{graph}, we mean a (possibly infinite) simplicial $1$-complex; 
in particular, we do not allow loops or multi-edges.
For a graph $\gam$, we denote its vertex set by $V(\gam)$ and its edge set by $E(\gam)$. 
One can define the \emph{right-angled Artin group $\aga$ with the underlying graph $\gam$} by the following presentation:
\[ \aga= \form{V(\gam) \vert [v,v']=1\textrm{ for each }\{v,v'\}\in E(\gam)}.\]

It is a fundamental fact that two right-angled Artin groups are isomorphic if and only if their underlying graphs are isomorphic~\cite{KMNR1980}, also \cite{Sabalka2009}, \cite{Koberda2011}. 
Note that $\aga$ is free abelian for a complete graph $\gam$, 
and free for a discrete graph $\gam$. 
In these two extreme cases, 
all the subgroups are right-angled Artin groups of the same type; 
namely, free abelian or free, respectively. 

Subgroups of more general right-angled Artin groups are diverse in their isomorphism types.
If a group $H$ embeds into another group $G$, we say $G$ \emph{contains} $H$ and write $H\le G$. By a \emph{long cycle}, we mean a cycle of length at least five.
If $\gam$ is a long cycle, then $\aga$ contains the fundamental group of a closed hyperbolic surface~\cite{SDS1989}.
Actually, the fundamental group of any closed surface with Euler characteristic less than $-1$ embeds into some right-angled Artin group~\cite{CW2004}.
A group $H$ is said to \emph{virtually embed} into another group $G$ if a finite index subgroup of $H$ embeds into $G$. It is an outstanding question whether the fundamental group of an arbitrary closed hyperbolic 3-manifold virtually embeds into some right-angled Artin group~\cite{HW2008,Agol2008}.
In this paper, we are mainly concerned about subgroup relations between distinct right-angled Artin groups.

\begin{que}\label{q:main}
Is there an algorithm to decide whether there exists an embedding between two given right-angled Artin groups?
\end{que}
We remark that the answer to Question~\ref{q:main} remains unchanged if we replace ``embedding'' with ``virtual embedding''. 
Indeed, suppose $\gam$ and $\Lambda$ are finite graphs and we have an embedding $H\to A(\gam)$, where $H\le A(\Lambda)$ is a finite index subgroup. 
Every finite index subgroup of $A(\Lambda)$ contains a copy of $A(\Lambda)$.  
To see this, let $V$ be the set of vertices of $\Lambda$, viewed as generators of $A(\Lambda)$, 
and let $H\le A(\Lambda)$ have finite index $N$.  Without loss of generality, $H$ is normal.
Then the group generated by $\{v^N\mid v\in V\}$ is contained in $H$ and is isomorphic to $A(\Lambda)$.

Let $\gam$ be a graph and $U$ be a set of vertices of $\gam$.
The \emph{induced subgraph  of $\gam$ on $U$} is the subgraph $\Lambda$ of $\gam$ consisting of $U$ and the edges whose endpoints are both in $U$;
we also say $\gam$ \emph{contains an induced} $\Lambda$ and write $\Lambda\le\gam$.
It is standard that $\Lambda\le\gam$ implies that $A(\Lambda)\le A(\gam)$.
By a \emph{clique of} $\gam$, we mean a complete subgraph of $\gam$.

\begin{defn}\label{d:extension graph}
Let $\gam$ be a graph.
\begin{enumerate}
	\item
The \emph{extension graph} of $\gam$ is the graph $\gex$ 
where the vertex set of $\gex$ consists of the words in $\aga$ that are conjugate to the vertices of $\gam$, and two vertices of $\gex$ are adjacent if and only if those two vertices commute, when considered as words in $\aga$.
\item
The \emph{clique graph} of $\gam$ is the graph $\gam_k$ such that the vertex set 
is the set of nonempty cliques of $\gam$
and two distinct cliques $K$ and $L$ of $\gam$ correspond to adjacent vertices of $\gam_k$ if and only if  $K$ and $L$ are both contained in some clique of $\gam$.
\end{enumerate}
\end{defn}

Our first main theorem describes a method of embedding a right-angled Artin group into $\aga$ using extension graphs.

\begin{thm}\label{t:thm1}
For finite graphs $\Lambda$ and $\gam$, $\Lambda\le\gex$ implies $A(\Lambda)\le\aga$.
\end{thm}

The next theorem limits possible embeddings between right-angled Artin groups. Actually, we will prove a stronger version of Theorem~\ref{t:thm2}; that is, Theorem~\ref{t:thm2 strong}.

\begin{thm}\label{t:thm2}
For finite graphs $\Lambda$ and $\gam$, $A(\Lambda)\le \aga$ implies $\Lambda\le\gex_k$.
\end{thm}

We remark that for each example of an embedding between right-angled Artin groups $A(\Lambda)$ and $A(\gam)$ known to the authors, there exists an inclusion $\Lambda\to\gex$.  We therefore conjecture:

\begin{con}\label{c:main}
For finite graphs $\Lambda$ and $\gam$, $A(\Lambda)\le A(\gam)$ if and only if $\Lambda\le\gex$.
\end{con}

Theorems~\ref{t:thm1} and~\ref{t:thm2} have a number of corollaries, 
many of which provide evidence for Conjecture \ref{c:main}.
Let $\gam$ be a graph and $L$ be a set of vertices in $\gam$.
Write $D_L(\gam)$ for the \emph{double} of $\gam$ along $L$; 
this means, $D_L(\gam)$ is obtained by taking two copies of $\gam$ and gluing them along the copies of the induced subgraph on $L$.
For a vertex $v$ of $\gam$, the \emph{link} of $v$ is the set of adjacent vertices of $v$ and denoted as  $\lk_\gam(v)$ or $\lk(v)$. The \emph{star} of $v$ is the set $\lk(v)\cup\{v\}$ and denoted as $\st_\gam(v)$ or $\st(v)$.  
An easy argument on HNN-extensions shows that for a graph $\gam$ and its vertex $t$,
$A(D_{\lk_\gam(t)}(\gam\smallsetminus\{t\}))\le A(\gam)$; see~\cite{LS2001,CSS2008}. 
We strengthen this result as follows.

\begin{cor}\label{c:double}
Let $\gam$ be a finite graph and $t$ be a vertex of $\gam$.
Then $A(D_{\st_\gam(t)}(\gam))\le A(\gam)$.
\end{cor}

Conjecture~\ref{c:main} has a complete answer is when $\Lambda$
is a \emph{forest}, namely, a disjoint union of trees. 
We first characterize right-angled Artin groups containing or contained in $A(P_4)$.

\begin{thm}\label{t:p4embed}
For a finite graph $\gam$, $A(P_4)\le A(\gam)$ implies that $P_4\le\gam$.
\end{thm}

\begin{thm}\label{t:forest in p4}
Any finite forest $\Lambda$ is an induced subgraph of $P_4^e$; in particular, $A(\Lambda)$ embeds into $A(P_4)$.
\end{thm}

\begin{cor}\label{c:forest}
Let $\Lambda$ and $\gam$ be finite graphs such that $\Lambda$ is a forest.
Then $A(\Lambda)\le\aga$ if and only if $\Lambda\le\gex$.
\end{cor}

We will use the shorthand $C_n$ to denote the graph which is a cycle of length $n$
and the shorthand $P_n$ for the graph which is a path on $n$ vertices (namely, of length $n-1$). 
We call $C_3$ a \emph{triangle}, and $C_4$ a \emph{square}.
When $\Lambda$ is a square, we will deduce an answer to Conjecture~\ref{c:main} from Theorem~\ref{t:thm1} in a stronger form as follows; the same result was originally proved by Kambites~\cite{Kambites2009}.

\begin{cor}[cf. \cite{Kambites2009}]\label{c:kambites}
For a finite graph $\gam$, $F_2\times F_2\cong A(C_4) \le \aga$ implies that $\gam$ contains an induced square.
\end{cor}

A graph $\gam$ is \emph{$\Lambda$--free} for some graph $\Lambda$ if $\gam$ contains no induced $\Lambda$. We prove Conjecture~\ref{c:main} when the target graph $\gam$ is triangle--free; this includes the cases when $\gam$ is bipartite or a cycle of length at least $4$.  It is easy to see that a graph $\gam$ is triangle--free if and only the associated Salvetti complex of $A(\gam)$ has dimension two.

\begin{thm}\label{t:triangle--free}
Suppose $\Lambda$ and $\gam$ are finite graphs such that $\gam$ is triangle--free.
Then $A(\Lambda)$ embeds in $A(\gam)$ if and only if $\Lambda$ is an induced subgraph of $\gex$.
\end{thm}

A corollary of Theorem \ref{t:triangle--free} is the following quantitative version of Conjecture~\ref{c:main} when the target graph is a cycle:

\begin{thm}\label{t:cntarget}
Let $m, n\geq 4$.
Then $A(C_m)\le A(C_n)$ if and only if $m=n+k(n-4)$ for some $k\ge0$.
\end{thm}

In particular, $A(C_5)$ contains $A(C_m)$ for every $m\ge6$. 
We write ${\gam}\opp$ for the \emph{complement graph} of a graph $\gam$;
this means ${\gam}\opp$ is given by completing $\gam$ and deleting the edges which occur in $\gam$.  The following result, originally due to the first author\footnote{The complement graph of $C_6$ was the first known example of a graph not containing a long induced cycle such that the corresponding right-angled Artin group contains the fundamental group of a closed hyperbolic surface~\cite{Kim2008}; see also~\cite{CSS2008,Kim2010}.}, is also an easy consequence of Theorem~\ref{t:thm1}.

\begin{cor}[cf. \cite{Kim2008}]\label{c:chordal}
For $n\ge4$, $A(C_{n-1}\opp)$ embeds into $A(C_n\opp)$;
in particular, $A(C_n\opp)$ contains $A(C_5)=A(C_5\opp)$ for any $n\ge6$.
\end{cor}

An interesting, but still unresolved case of Conjecture~\ref{c:main} is when $\Lambda$ is a long cycle.
A graph $\gam$ is called \emph{weakly chordal} if $\gam$ does not contain an induced $C_n$ or $C_n^{opp}$ for $n\geq 5$.
We will show that $\gex$ contains an induced long cycle if and only if $\gam$ is not weakly chordal (Lemma~\ref{l:cmfree}).
Hence, Conjecture~\ref{c:main} for the case when $\Lambda$ is a long cycle reduces to the following:

\begin{con}[Weakly--Chordal Conjecture]\label{c:wchordal}
If $\gam$ is a weakly chordal graph, then $A(\gam)$ does not contain $A(C_n)$ for any $n\geq 5$.
\end{con}

We note the Weakly--Chordal Conjecture is true when $\gam$ is triangle--free or square--free; see Corollary \ref{c:tsfree}.

A final topic which we address in this paper is the (non)--existence of \emph{universal} right-angled Artin groups.
A right-angled Artin group is \emph{$n$--dimensional} if its cohomological dimension is $n$.
An $n$--dimensional finitely generated right-angled Artin group $G$ is called a \emph{universal $n$-dimensional right-angled Artin group} if $G$ contains copies of each $n$--dimensional right-angled Artin group.  Since $F_2$ contains every other finitely generated free group and since free groups are precisely the groups with cohomological dimension one, $F_2$ is a universal $1$--dimensional right-angled Artin group.
We prove:

\begin{thm}\label{t:universal}
There does not exist a universal two--dimensional right-angled Artin group.
\end{thm}

This paper is organized as follows. 
In Section~\ref{sec:preliminary}, we recall basic facts on right-angled Artin groups and mapping class groups. 
Definition and properties of extension graphs will be given in Section~\ref{sec:gex}.
We prove Theorems~\ref{t:thm1} and~\ref{t:thm2} in Section~\ref{sec:gex and raag}.
Section~\ref{sec:trees} mainly discusses $A(P_4)$ and as a result, we prove Conjecture~\ref{c:main} when $\Lambda$ or $\gam$ is a forest.
Answers to Conjecture~\ref{c:main} for complete bipartite graphs will be given in Section~\ref{sec:complete bipartite}.
Results on co-contraction are proved in Section~\ref{sec:cocontraction}.
In Sections~\ref{sec:long cycle}, we prove Conjecture~\ref{c:main} for $\gam$ a triangle--free graph, and give the quantitative version when $\gam$ and $\Lambda$ are both long cycles.  In Section \ref{s:universal}, we prove that there is no universal two--dimensional right-angled Artin group.
Section~\ref{sec:appendix} contains a topological proof of Corollary~\ref{c:double}.

\section{Background material}\label{sec:preliminary}
\subsection{Centralizers of right-angled Artin groups}
Suppose $\gam$ is a graph. Each element in $V(\gam)\cup V(\gam)^{-1}$ is called a \emph{letter}. Any element in $A(\gam)$ can be expressed as a \emph{word},
which is a finite multiplication of letters.
Let $w=a_1\ldots a_l$ be a word in $A(\gam)$ where $a_1,\ldots,a_l$ are letters.
We say $w$ is \emph{reduced} if any other word representing the same element in $A(\gam)$ as $w$ has at least $l$ letters.
We say $w$ is \emph{cyclically reduced} if any cyclic conjugation of $w$ is reduced.
The \emph{support} of a reduced word $w$ is the smallest subset $U$ of $V(\gam)$ such that each letter of $w$ is in $U\cup U^{-1}$; we write $U=\supp(w)$.

We will use the notation $h^{-1}gh=g^h$ for two group elements $g$ and $h$.  Let $g$ be an element in $A(\gam)$.
There exists $p\in A(\gam)$ such that $g=w^p$ 
for some cyclically reduced word $w$.
Put $B=\supp(w)$, and let $B_1,B_2,\ldots,B_k$ be the vertex sets of the connected components of $\gam_B^{opp}$, where $\gam_B$ is the induced subgraph of $\gam$ on $B$.   
Then one can write
$w= g_1^{e_1} g_2^{e_2}\cdots g_k^{e_k}$ where $\supp(g_i)\in B_i$, $e_i > 0$ and $\form{g_i}$ is maximal cyclic  for each $i=1,\ldots,k$. We say each $g_i$ is a \emph{pure factor} of $g$, and the expression $g = p^{-1}g_1^{e_1}\cdots g_k^{e_k}p$ is a \emph{pure factor decomposition} of $g$; this decomposition is unique up to reordering~\cite{Servatius1989}.
In the special case when $p=1=k$ and $e_1=1$,  $g$ is a pure factor.
The centralizer of a word in $A(\gam)$ is completely described by the following theorem.

\begin{thm}[Centralizer Theorem~\cite{Servatius1989}]\label{t:centralizer}
Let $\gam$ be a finite graph and $g = p^{-1}g_1^{e_1}\cdots g_k^{e_k}p$ be a pure factor decomposition of $g$ in $\aga$.
Then any element $h$ in the centralizer of $g$ can be written as 
\[h=p^{-1}g_1^{f_1}\cdots g_k^{f_k} g' p\]
 for some integers $f_1,\ldots,f_k$ and $g'\in \aga$ such that each vertex in $\supp(g')$ is adjacent to every vertex in
 $\cup_i \supp(g_i)$ in $\gam$.
\end{thm}

For two graphs $\gam_1$ and $\gam_2$, the \emph{join} $\gam_1\ast\gam_2$ of $\gam_1$ and $\gam_2$ is defined by the relation
$\gam_1\ast\gam_2= (\gam_1\opp\coprod\gam_2\opp)\opp$.
Note that $A(\gam_1\ast\gam_2) = A(\gam_1)\times A(\gam_2)$.
A graph is said to \emph{split as a nontrivial join} if it can be written as $\gam_1\ast\gam_2$ for nonempty graphs $\gam_1$ and $\gam_2$. 
As a corollary of Theorem~\ref{t:centralizer}, one can describe when the centralizer of a reduced word is non-cyclic.

\begin{cor}[\cite{BC2010}]
Let $g\in A(\gam)$ be cyclically reduced.  The following are equivalent:
\begin{enumerate}
\item
The centralizer of $g$ is noncyclic.
\item
The support of $g$ is contained in a non-trivial join of $\gam$.
\item
The supports of the words in the centralizer of $g$ is contained in a non-trivial join of $\gam$.
\end{enumerate}
\end{cor}

The following is well-known and stated at various places, such as~\cite{CV2009}.

\begin{lem}\label{l:rank}
For a finite graph $\Gamma$, the maximum rank of a free abelian subgroup of $A(\Gamma)$ is the size of a largest clique in $\Gamma$.
\end{lem}

\subsection{Mapping class groups}
The material in the subsection can be found in most general references on mapping class groups, such as Farb and Margalit's book \cite{FM2011}.  Let $\Sigma$ be a surface of genus $g\geq 0$ and $n\geq 0$ punctures.  The mapping class group of $\Sigma$ is defined to be the group of components of the group of orientation--preserving homeomorphisms of $\Sigma$.  Precisely, \[\Mod(\Sigma)\cong \pi_0(\Homeo^+(\Sigma)).\]  It is a standard fact that mapping class groups are finitely presented groups.
The following result is due to Nielsen and Thurston:

\begin{thm}
Let $\psi\in\Mod(\Sigma)$.  Then exactly one of the following three possibilities occurs:
\begin{enumerate}
\item
The mapping class $\psi$ has finite order in $\Mod(\Sigma)$.
\item
The mapping class $\psi$ has infinite order in $\Mod(\Sigma)$ and preserves some collection of homotopy classes of essential, non--peripheral, simple closed curves.  In this case, $\psi$ is called \emph{reducible}.
\item
The mapping class $\psi$ has infinite order in $\Mod(\Sigma)$ and for each essential, non--peripheral, simple closed curve $c$ and each nonzero integer $N$, $\psi^N(c)$ and $c$ are not homotopic to each other.  In this case, $\psi$ is called \emph{pseudo-Anosov}.
\end{enumerate}
\end{thm}

In many senses, ``typical" mapping classes are pseudo-Anosov.  Nevertheless, we shall be exploiting reducible mapping classes in this paper.  The following result characterizes typical reducible mapping classes, and is due to Birman, Lubotzky and McCarthy in \cite{BLM1983}:
  
\begin{thm}
Let $\psi$ be a reducible mapping class in $\Mod(\Sigma)$.  Then there exists a multicurve $\mC$, called a canonical reduction system, and a positive integer $N$ such that:
\begin{enumerate}
\item
Each element $c\in\mC$ is fixed by $\psi^N$.
\item
The restriction of $\psi^N$ to the interior of each component of $\Sigma\smallsetminus\mC$ is either trivial or pseudo-Anosov.
\end{enumerate}
The mapping class $\psi^N$ is called \emph{pure}.
\end{thm}

If $\psi$ is a pure mapping class and $\mC$ is its canonical reduction system, $\psi$ may not act trivially on a neighborhood of $\mC$.  In particular, $\psi$ may perform Dehn twists, which are homeomorphisms given by cutting $\Sigma$ open along $\mC$ and re--gluing with a power of a twist.

The primary reducible mapping classes with which we concern ourselves in this paper are Dehn twists and pseudo-Anosov homeomorphisms supported on a proper, connected subsurface of $\Sigma$. We denote by $T_\alpha$ the Dehn twist about a simple closed curve $\alpha$ on a surface; here, we assume $\alpha$ is oriented whenever needed.

\begin{lem}[cf. \cite{Penner1988}, \cite{Mangahas2010}]\label{l:filling curves}
Suppose $\Sigma$ is a connected surface and $\alpha_1,\ldots,\alpha_r$ are pairwise--non-isotopic, essential simple closed curves on $\Sigma$ such that $\cup_i \alpha_i$ is connected.
We let $\Sigma_0$ be the connected subsurface of $\Sigma$ obtained by taking a regular neighborhood of $\cup_i \alpha_i$ and capping off boundary curves which are nullhomotopic in $\Sigma$.  
For a multi--index  $\underline{n}$ of nonzero integers $(n_1,\ldots,n_r)$, 
let $\psi_{\underline{n}}$ be the product of the Dehn twists given by
 \[T_{\alpha_1}^{n_1}\cdots T_{\alpha_r}^{n_r}.\]  Then there exists a multi--index $\underline{n}$ satisfying the following:
\begin{enumerate}[(i)]
\item
The restriction of $\psi^m_{\underline{n}}$ on $\Sigma_0$ is pseudo-Anosov for any $m\neq 0$;
\item
For any $i,j=1,\ldots,r$, $\psi^m_{\underline{n}}(\alpha_i)$ essentially intersects $\alpha_j$ for $m$ sufficiently large;
\item
If an essential subsurface $\Sigma_1$ of $\Sigma$ essentially intersects $\Sigma_0$,
$\psi_1$ is a pure mapping class on $\Sigma$ which is pseudo-Anosov only on $\Sigma_1$,
and $\beta$ is an essential simple closed curve in $\Sigma_1$,
then each $\psi^m_{\underline{n}}(\alpha_i)$ essentially intersect $\psi_1^{k}(\beta)$ for $m,k$ sufficiently large.
\end{enumerate}
\end{lem}
Such a pseudo-Anosov homeomorphism can be constructed quite explicitly:
\begin{proof}[Proof of Lemma \ref{l:filling curves}]
The proof is by induction on $r$.  If $\alpha_1$ and $\alpha_2$ are intersecting simple closed curves then there exist powers $n$ and $m$ such that $T_{\alpha_1}^nT_{\alpha_2}^m$ is pseudo-Anosov on the subsurface filled by $\alpha_1$ and $\alpha_2$.  For the inductive step, let $\psi$ be a pseudo-Anosov supported on a surface $\Sigma_0'$ and let $\alpha$ be a curve which intersects $\Sigma_0'$ essentially (in the sense that it cannot be homotoped off of $\Sigma_0'$).  Then by \cite{Mangahas2010} there exist powers $n$ and $m$ such that $\psi^n T_{\alpha}^m$ is pseudo-Anosov on the subsurface filled by $\Sigma_0'$ and $\alpha$.
For (iii), apply  \cite{Mangahas2010} to $\psi_1^{-k} \psi^m_{\underline{n}}$. 
\end{proof}

By the \emph{disjointness} of two curves or two subsurfaces of a surface, we will mean disjointness within their isotopy classes. Namely, $c_1$ and $c_2$ are not disjoint if no isotopy representatives of $c_1$ and $c_2$ are disjoint.

\begin{defn}\label{d:coincidence}
Let $\Sigma$ be a connected surface and $C$ be either 
\begin{enumerate}[(i)]
\item
a collection of essential simple closed curves on $\Sigma$; or,
\item
a collection of connected essential subsurfaces of $\Sigma$.
\end{enumerate}
Then the \textit{co--incidence graph} of $C$ is a graph where $C$ is the vertex set 
and two vertices $x$ and $y$ are adjacent if and only if $x$ and $y$ are disjoint.
\end{defn}

\section{The topology and geometry of extension graphs}\label{sec:gex}
Let $\gam$ be a finite graph.
Note that there is a natural retraction $\gex\to\gam$ which maps $v^w$ to $v$ for each vertex $v$ of $\gam$ and a word $w$ of $\aga$.
Suppose $a$ and $b$ are vertices of $\gam$, and $x$ and $y$ are words in $\aga$.
By Theorem~\ref{t:centralizer}, $a^x$ and $b^y$ commute in $\aga$ if and only if $[a,b]=1$ and $a^w=a^x,b^w=b^y$ for some $w\in\aga$; this is equivalent to $\form{\st(a)}x\cap \form{\st(b)}y\ne \varnothing$.
We will think of $\gam$ and $\gex$ as metric graphs so that adjacent vertices have distance one. We will denote the distance functions commonly as $d(\cdot,\cdot)$. There is a natural right action of $A(\gam)$ on $\gex$; namely, for $g,w\in\aga$ and $v\in V(\gam)$ we define $v^w.g= v^{wg}$.
Note that the quotient of $\gex$ by this action is $\gam$.
In this paper, we will always regard $\gam$ as an induced subgraph of $\gex$ and $\gam_k$.
In particular, $\gex$ is the union of conjugates of $\gam$.

We can explicitly build the graph $\gex$ as follows. Fix a vertex $v$ of $\gam$ and glue two copies of $\gam$ along the star of $v$; the resulting graph is an induced subgraph of $\gex$ on $V(\gam)\cup V(\gam)^v$. To obtain $\gex$, we repeat this construction for each vertex of $\gam$ countably many times, at each finite stage getting a (usually) larger finite graph. So, we have the following.

\begin{lem}\label{l:double}
Let $\gam$ be a finite graph and $\Lambda$ be a finite induced subgraph of $\gex$.
Then there exists an $l>0$, a sequence of vertices $v_1,v_2,\ldots,v_l$ of $\gex$, and a sequence of finite induced subgraphs $\gam=\gam_0\le\gam_1\le\ldots\le\gam_l$ of $\gex$ where $\gam_i$ is obtained by taking the double of $\gam_{i-1}$ along $\st_{\gam_{i-1}}(v_i)$ for each $i=1,\ldots,l$, such that $\Lambda\le\gam_l$.\qed
\end{lem}

There are various other ways to think of $\gex$.  Probably the most useful of these is that $\gex$ is a  ``universal" graph obtained from $\gam$, in the sense that $\gex$ produces all potential candidates for right-angled Artin subgroups of $A(\gam)$ (cf. Theorems \ref{t:thm1} and \ref{t:thm2}).  Another useful perspective is that $\gex$ is an analogue of the complex of curves for $A(\gam)$ (cf. Lemma~\ref{l:extension} (4)).  In this section we will list some of the properties of $\gex$ and amplify these perspectives.

The essential tool in studying extension graphs is the following special case of a result due to the second author in \cite{Koberda2011}:

\begin{lem}[\cite{Koberda2011}]\label{l:mod}
Let $c_1,\ldots,c_m$ be pairwise--non-isotopic, essential, simple closed curves on a surface $\Sigma$, and let $T_1,\ldots,T_m$ be the respective Dehn twists in the mapping class group $\Mod(\Sigma)$.  Let $\gam$ be the co--incidence graph of $\{c_1,\ldots,c_m\}$, so that the vertices corresponding to two curves are connected if and only if the two curves are disjoint.  Then there exists an $N>0$ such that for any $n\geq N$, \[\langle T_1^n,\ldots,T_m^n\rangle \cong A(\gam)<\Mod(\Sigma).\]  Furthermore, for any finite graph $\gam$ one can find a surface and a collection of curves with co--incidence graph $\gam$.
\end{lem}

We note that there is a similar construction which is an equivalent tool for our purposes, using pseudo-Anosov homeomorphisms supported on subsurfaces of a given surface $\Sigma$ instead of Dehn twists.  M. Clay, C. Leininger and J. Mangahas have recently proven in \cite{CLM2010} that under certain further technical conditions, powers of such mapping classes generate a quasi--isometrically embedded right-angled Artin group:

\begin{lem}[\cite{CLM2010}]\label{l:CLM}
Let $\psi_1,\ldots,\psi_m$ be pseudo-Anosov homeomorphisms supported on subsurfaces $\Sigma_,\ldots,\Sigma_m$ such that no inclusion relations hold between $\Sigma_i$ and $\Sigma_j$ for $i\neq j$.
\begin{enumerate}
\item
There exists an $N$ such that for each $n\geq N$, the group $\langle\psi_1^n,\ldots,\psi_m^n\rangle$ is a right-angled Artin group which is quasi--isometrically embedded in the mapping class group.
\item
The abstract isomorphism type of $\langle\psi_1^n,\ldots,\psi_m^n\rangle$ is the ``expected" right-angled Artin group, as in Lemma \ref{l:mod}.
\item
Each nontrivial word in the group $\langle\psi_1^n,\ldots,\psi_m^n\rangle$ is pseudo-Anosov on the minimal subsurface filled by the letters occurring in the word.
\end{enumerate}
\end{lem}


The following lemma is sometimes called \textit{Manning's bottleneck criterion}. Recall that in a metric space $(X,d)$, a point $m$ is a \textit{midpoint} of two points $x$ and $y$ if $d(x,m)=d(y,m)=\frac12 d(x,y)$. We denote by $B(x;r)$ the $r$--ball centered at $x$.
\begin{lem}[\cite{Manning2005}]\label{l:bottleneck}
A geodesic metric space $(X,d)$ is quasi-isometric to a tree if and only if there is a $\Delta>0$ satisfying the following:
for every pair of points $x,y$ in $X$, there is a midpoint $m$ of $x$ and $y$ such that every path from $x$ to $y$ intersects the $B(m;\Delta)$.
\end{lem}

We summarize important geometric properties of $\gex$ as follows.

\begin{lem}\label{l:extension}
Let $\gam$ be a finite graph.
\begin{enumerate}
\item
If $\gam=\gam_1*\gam_2$ for some finite graphs $\gam_1$ and $\gam_2$, then $\gex=\gam_1^e*\gam_2^e$.
\item
If $\gam=\gam_1\coprod\gam_2$ for some finite graphs $\gam_1$ and $\gam_2$, then $\gex$ is a countable union of disjoint copies of $\gam_1^e$ and $\gam_2^e$.
\item
The operation $\gam\to\gex$ respects induced subgraphs: if $\Lambda\subset\gam$ is an induced subgraph then $\Lambda^e\subset\gex$ is an induced subgraph.
\item
The graph $\gex$ embeds as an induced subgraph into the $1$-skeleton of the curve complex of some surface $\Sigma$, which we can take to be closed of some genus $g$ depending on $\gam$.
\item
Suppose $\gam$ is connected.  Then the graph $\gex$ is connected.  The graph $\gex$ has finite diameter if and only if $\gam$ is an isolated vertex or $\gam$ splits as a nontrivial join.  The graph $\gex$ is finite if and only if $\gam$ is complete.
\item
Suppose $\gam$ is connected.
If two vertices $w,w'$ of $\gex$ do not belong to the same conjugate of $\gam$, then $w$ and $w'$ are separated by the star of some vertex.
If we further assume that $\gam$ has no central vertices, then the star of each vertex $v$ of $\gex$ separates $\gex$. 
\item
Suppose $\gam$ is connected.
Then $\gex$ is quasi--isometric to a tree.
\item
The chromatic number of $\gam$ is equal to that of $\gex$.
\end{enumerate}
\end{lem}

\begin{proof}
(1)
Since $A(\gam)=A(\gam_1)\times A(\gam_2)$, the group $A(\gam_1)$ acts trivially on the vertices of $\gam_2$, and vice versa. 
So each conjugate of every vertex of $\gam_2$ will be adjacent to each conjugate of every vertex of $\gam_1$. We thus get the claimed splitting.

(2)
We get a splitting $A(\gam)=A(\gam_1)*A(\gam_2)$.  No conjugate of any vertex of $\gam_1$ is adjacent to any conjugate of any vertex of $\gam_2$ by the definition of $\gex$.  The description of $\gex$ can be seen by taking one copy of $\gam_1^e$ for each element of $A(\gam_2)$ and one copy of $\gam_2^e$ for each element of $A(\gam_1)$.  The conjugation action permutes the copies of these graphs according to the regular representation.

(3)
We clearly have a map $\Lambda^e\to\gex$.  Any edge between two vertices of $\gex$ is a conjugate of an edge in $\gam$ by an element of $A(\gam)$.  The claim follows, since two vertices in the image of $\Lambda^e$ are connected by an edge if and only if those two vertices are simultaneously conjugate to adjacent vertices in $\gam$, be it by an element of $A(\Lambda)$ or an element of $A(\gam)$.

(4)
One can realize $\gam$ as the co--incidence graph of pairwise--non-isotopic simple closed curves $\{\alpha(v):v\in V(\gam)\}$ on a closed surface $\Sigma$. 
Namely, $\alpha$ is an embedding from $\gam$ to the curve complex $\mC(\Sigma)$ such that the image of $\alpha$ is an induced subgraph. 
By Lemma~\ref{l:mod}, there exists an $M>0$ such that the map $\phi:A(\gam)\to\Mod(\Sigma)$ defined by $\phi(v)=T_{\alpha(v)}^N$ is injective.
We extend $\alpha$ to $\alpha^e:\gex\to\mC(\Sigma)$ as follows:
\[ \alpha^e(v^w)=\phi(w^{-1}).\alpha(v),\mbox{ for }v\in V(\gam)\mbox{ and }w\in A(\gam).\]
We claim that $\alpha^e$ is an injective graph map such that the image is an induced subgraph.
To check that $\alpha^e$ is well-defined, suppose $v^w=v'^{w'}$ for some $v,v'\in V(\gam)$ and $w,w'\in A(\gam)$. Then $v=v'$ and $w'w^{-1}\in\form{\st(v)}$ and so, $\phi(w'w^{-1})$ is some multiplication of Dehn twists about $\alpha(v)$ and simple closed curves which are disjoint from $\alpha(v)$. It follows that $\alpha(v)=\phi(w'w^{-1}).\alpha(v)$, and $\alpha^e(v^w)=\alpha^e(v'^{w'})$. We can also easily see that $\alpha^e$ maps an edge $\{v^w,v'^w\}$ to an edge, where $w\in A(\gam)$ and $\{v,v'\}\in E(\gam)$. For injectivity, assume $\alpha^e(v^w)=\alpha^e(v'^{w'})$. After conjugation, we may assume $w'=1$.
Then $\phi(v')=T_{\alpha(v')}^N = T_{\phi(w^{-1}).\alpha(v)}^N=\phi(w^{-1})\circ T_{\alpha(v)}^N\circ \phi(w)
=\phi(v^w)$. The injectivity of $\phi$ implies that $v'=v^w$.
Similar argument shows that the image of $\alpha^e$ is an induced subgraph.

To alternatively see that every \emph{finite} subgraph of $\gex$ embeds in the curve complex of some surface as an induced subgraph, one can use the result of Clay, Leininger and Mangahas in Lemma~\ref{l:CLM} to embed $A(\gam)$ into the mapping class group of some surface $\Sigma$ by sending vertices to certain pseudo-Anosov homeomorphisms on connected subsurfaces of $\Sigma$.  By considering conjugates of sufficiently high powers of these pseudo-Anosovs by elements of $A(\gam)$, we obtain a collection of subsurfaces of $\Sigma$ whose co--incidence graph is precisely $\gex$.  These subsurfaces are then equipped with pseudo-Anosov homeomorphisms, any finite collection of which generates a right-angled Artin subgroup of $A(\gam)$ corresponding to a finite subgraph $\Lambda$ of $\gex$. Approximating the stable laminations of these pseudo-Anosov homeomorphisms with simple closed curves on $\Sigma$ embeds any finite subgraph of $\gex$ into the curve complex of $\Sigma$ as an induced subgraph.

(5)
The first claim is obvious by the construction of $\gex$ via iterated doubles along stars of vertices, namely Lemma~\ref{l:double}.

If $\gam$ splits nontrivially as $\gam_1*\gam_2$ then every conjugate of each vertex in $\gam_1$ is adjacent to every conjugate of each vertex in $\gam_2$, so that $\gex$ has finite diameter.  If $\gam$ is complete then the conjugation action of $A(\gam)$ on the vertices of $\gam$ is trivial so that $\gex=\gam$.  Conversely, if $\gam$ is not complete then there are two vertices in $\gam$ which generate a copy of $F_2$ in $A(\gam)$ and hence there are infinitely many distinct conjugates of these vertices.  If $\gam$ does not split as a join then we can represent $A(\gam)$ as powers of Dehn twists about simple closed curves on a connected surface $\Sigma$ or as pseudo-Anosov homeomorphisms on connected subsurfaces of $\Sigma$ (see \cite{CLM2010} or \cite{Koberda2011}), and the statement that $\gam$ does not split as a nontrivial join is precisely the statement that these curves fill a connected subsurface $\Sigma_0$ of $\Sigma$.  By Lemma~\ref{l:CLM}, there is a word in the powers $\psi$ of these Dehn twists which is pseudo-Anosov on $\Sigma_0$ and hence has a definite translation distance in the curve complex of $\Sigma_0$.  
It follows that $\psi$--conjugates of twisting curves in the generators of $A(\gam)$ have arbitrarily large distance in the curve complex of $\Sigma_0$.  
There is a map $\phi$ from the graph $\gex$ to the curve complex of $\Sigma_0$ which sends a vertex to the curve about which the vertex twists.  General considerations show that this map is distance non--increasing.  It follows that $\gex$ has infinite diameter.

(6)
By the symmetry of conjugate action, we may assume $w$ belongs to $\gam$.
Construct $\gex$ as \[\gam=\gam_0\subset\gam_1\subset\cdots,\] where the union of these graphs is $\gex$ and $\gam_n$ is obtained from $\gam_{n-1}$ by doubling $\gam_{n-1}$ along the star of a 
vertex of $\gam$, for each $n$.
There is $k$ such that $w'$ is a vertex of $\gam_k\smallsetminus\gam_{k-1}$;
we choose $\gam_0,\gam_1,\ldots,$ such that $k$ is minimal.
This implies that $w$ and $w'$ belongs to distinct components of $\gam_k\smallsetminus\st(v)$ for some vertex $v$ of $\gam_{k-1}$. We claim that $w$ and $w'$ remain separated in $\gex\smallsetminus\st(v)$.

If $w$ and $w'$ are in the same component of $\gex\smallsetminus\st(v)$, then this fact becomes evident at a finite stage of the construction of $\gex$.  There exists 
$m\ge k$ 
such that $w$ and $w'$ are separated in $\gam_m\smallsetminus\st(v)$; and furthermore, we assume $\gam_{m+1}$ is the double of $\gam_m$ along the star of a vertex $z$, and $w$ and $w'$ are in the same component of $\gam_{m+1}\smallsetminus\st(v)$. See Figure~\ref{f:extension6} (a). Note that if $z\notin\st(v)$ then $\st(z)=\st_{\gam_m}(z)$ cannot intersect both of the components of $\gam_m\smallsetminus\st(v)$ which contain $w$ and $w'$ and thus those two vertices are in two different components of $\gam_{m+1}\smallsetminus\st(v)$.  
Therefore, we may assume $z\in\st(v)$.  In that case, the two copies of $v$ in both copies of $\gam_m$ are identified, so that the star of $v$ in $\gam_{m+1}$ is the union of the two stars of $v$ in the two copies of $\gam_m$; see Figure~\ref{f:extension6} (b).  The stars of $v$ separated both copies of $\gam_m$ into components $S_1,\ldots,S_n$ and $T_1,\ldots,T_n$, where these are subgraphs of the two respective copies of $\gam_m$.  It is possible that $S_i$ is glued to $T_i$ along some common vertices, but it is not possible for two distinct components $T_i$ and $T_j$ to be glued to a single copy of $S_i$.  Indeed otherwise $S_i$ and $S_j$ would share a common vertex, a contradiction.  It follows that the components of $\gam_m\smallsetminus\st(v)$ which contain $w$ and $w'$ are not contained in the same component of $\gam_{m+1}\smallsetminus\st(v)$.

To prove the second claim, assume $v,w\in\gam$ and simply let $w'$ be a vertex in the double of $\gam$ along $\st(v)$ such that $w'\not\in\gam$. We have shown that $w$ and $w'$ belong to distinct components of $\gex\smallsetminus\st(v)$.

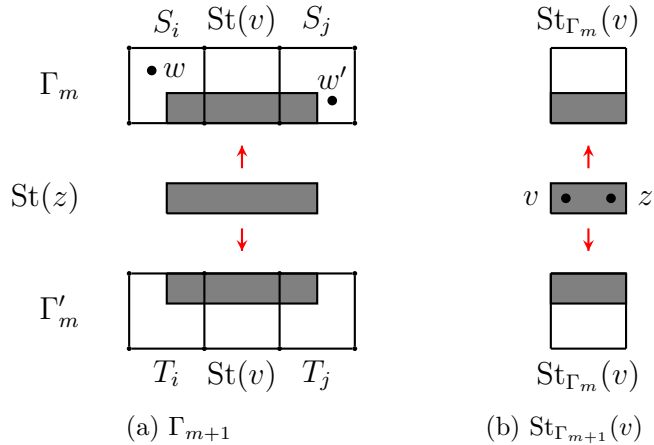
\begin{figure}[htb!]
  \tikzstyle {a}=[red,postaction=decorate,decoration={%
    markings,%
    mark=at position 1 with {\arrow[red]{stealth};}}]
  \tikzstyle {b}=[blue,postaction=decorate,decoration={%
    markings,%
    mark=at position .43 with {\arrow[blue]{stealth};},%
    mark=at position .57 with {\arrow[blue]{stealth};}}]
  \tikzstyle {v}=[draw,shape=circle,inner sep=0pt]
  \tikzstyle {bv}=[black,draw,shape=circle,fill=black,inner sep=1pt]
  \tikzstyle{every edge}=[-,draw]
\subfloat[(a) $\gam_{m+1}$]{
	\begin{tikzpicture}[thick]
\draw [fill=gray] (.5,1) -- (.5,.6) -- (2.5,.6) -- (2.5,1) -- (.5,1);
\draw [fill=gray] (.5,3) -- (.5,3.4) -- (2.5,3.4) -- (2.5,3) -- (.5,3);
\draw [fill=gray] (.5,1.8) -- (.5,2.2) -- (2.5,2.2) -- (2.5,1.8) -- (.5,1.8);
\foreach \i in {0,...,3}
    \foreach \j in {0,1,3,4}
		\draw (\i, \j ) node [v] (v\i\j) {};
\foreach \j in {0,1,3,4}
\draw (v0\j)--(v3\j);
\foreach \i in {0,...,3}
\draw (v\i0)--(v\i1);
\foreach \i in {0,...,3}
\draw (v\i3)--(v\i4);
\draw[a] (1.5, 2.4) -- (1.5, 2.7);
\draw[a] (1.5, 1.6) -- (1.5, 1.3);
\draw (.5, 4) node [above] {$S_i$};
\draw (1.5, 4) node [above] {$\st(v)$};
\draw (1.5, 0) node [below] {$\st(v)$};
\draw (2.5, 4) node [above] {$S_j$};
\draw (.5, 0) node [below] {$T_i$};
\draw (2.5,0) node [below] {$T_j$};
\draw (-.5,3.5) node [left] {$\gam_m$};
\draw (-.5,0.5) node [left] {$\gam_m'$};
\draw (-.5,2) node [left] {$\st(z)$};
\draw (.3,3.7) node[bv] {};
\draw (.3,3.7) node[right] {${w}$};
\draw (2.7,3.3) node[bv] {};
\draw (2.7,3.3) node[above] {${w'}$};
\end{tikzpicture}
}
\hspace{.4in}
\subfloat[(b) $\st_{\gam_{m+1}}(v)$]{
	\begin{tikzpicture}[thick]
\draw [fill=gray] (1,1) -- (1,.6) -- (2,.6) -- (2,1) -- (1,1);
\draw [fill=gray] (1,3) -- (1,3.4) -- (2,3.4) -- (2,3) -- (1,3);
\draw [fill=gray] (1,1.8) -- (1,2.2) -- (2,2.2) -- (2,1.8) -- (1,1.8);
\draw node [] (.5,0) {};
\draw node [] (2.5,0) {};
\foreach \j in {0,1,3,4}
\draw (1,\j) -- (2,\j);
\foreach \i in {1,2}
\draw (\i,0)--(\i,1);
\foreach \i in {1,2}
\draw (\i,3)--(\i,4);
\draw[a] (1.5, 2.4) -- (1.5, 2.7);
\draw[a] (1.5, 1.6) -- (1.5, 1.3);
\draw (1.5, 4) node [above] {$\st_{\gam_m}(v)$};
\draw (1.5, 0) node [below] {$\st_{\gam_m}(v)$};
\draw (1.2,2) node[bv] {};
\draw (1,2) node[left] {${v}$};
\draw (1.8,2) node[bv] {};
\draw (2,2) node[right] {${z}$};

	\end{tikzpicture}
}
  \caption{Proof of Lemma~\ref{l:extension} (6).}
  \label{f:extension6}
\end{figure}

(7)
From (5), we can assume that $\gam$ does not split as a nontrivial join.
To apply Lemma~\ref{l:bottleneck}, let us consider two vertices $x_0,y_0$ in $\gex$, a geodesic $\gamma_0$ joining them, and the midpoint $m$ of $\gamma_0$.
Here, $m$ is either a vertex or the midpoint of an edge.
We have only to find $\Delta>0$ such that $B(m;\Delta)$ separates $x_0$ and $y_0$.
For arbitrary pair of points $p,q$ on $\gamma_0$, we denote by $[p,q]$ the geodesic on $\gamma_0$ joining $p$ and $q$.
We may assume $D=\diam\gam\ge3$ and $d(x_0,y_0)\ge3D$.
We now inductively define $z_i, x_{i+1},y_{i+1}$ and $\gamma_{i+1}$ as long as $d(x_i,y_i)\ge3D$,
for $i\ge0$:
\begin{enumerate}[(i)]
\item
using Lemma~\ref{l:2D} (2) below, choose a vertex $z_i$ such that $\st(z_i)$ separates $x_i$ from $y_i$, and $d(z_i,x_i),d(z_i,y_i)\ge D$;
\item
$x_{i+1}$ is a vertex in $\st(z_i)\cap\gamma_i$;
\item $\gamma_{i+1}$ is the closure of the component of $\gamma_i\smallsetminus x_{i+1}$ containing $m$;
\item
$\partial\gamma_{i+1} = \{x_{i+1},y_{i+1}\}$.
\end{enumerate}

Note that $d(x_i,y_i)=l(\gamma_i)$ is strictly decreasing, since $x_{i+1}\not\in\{x_i,y_i\}$.
So for some $j>0$, we have $d(x_j,y_j)\ge 3D$ and $d(x_{j+1},y_{j+1}) < 3D$. 
Without loss of generality, let us assume $x_0,x_j,x_{j+1},y_j=y_{j+1},y_0$ appear on $\gamma_0$ in this order.
If $p$ is a vertex in $\st(z_j)\cap[x_0,x_j]$, then $d(x_j,x_{j+1})  < d(p,x_{j+1}) \le d(p,z_j)+d(z_j,x_{j+1})\le2$ and this contradicts to $d(x_j,x_{j+1})\ge d(x_j,z_j) - 1 \ge D-1$.
So, $\st(z_j)$ intersects neither $[x_0,x_j]$ nor $[y_j,y_0]$.
If there were a path $\delta$ from $x_0$ to $y_0$ not intersecting $\st(z_j)$, then $\delta\cup[x_0,x_j]\cup [y_j,y_0]$ would be a path joining $x_j$ to $y_j$ without intersecting $\st(z_j)$. It follows that $\st(z_j)$ separates $x_0$ from $y_0$.
We see that $3D+2$ is a desired value for $\Delta$; for, $m\in[x_{j+1},y_{j+1}]$ and 
 $\st(z_j) = B(z_j;1) \subseteq B(x_{j+1};2)\subseteq B(m;3D+2)$.
 
(8)
 A coloring of $\gam$ pulls back to a coloring of $\gex$ by the natural retraction $\gex\to\gam$.
Hence, the chromatic number of $\gex$ is at most that of $\gam$.  The converse is immediate.
\end{proof}

\begin{lem}\label{l:2D}
Let $\gam$ be a finite graph with diameter $D$, and $x,y\in V(\gex)$.
\begin{enumerate}
\item
There exist $x = x_0,x_1,\ldots,x_l,x_{l+1}=y$ in $V(\gex)$ such that
\begin{enumerate}[(i)]
\item
$x_i$ and $x_{i+1}$ belong to the same conjugate of $\gam$ for $i=0,1,\ldots,l$,
\item
$\st(x_i)$ separates $x$ from $y$, for $i=1,\ldots,l$.
\end{enumerate}
\item
Suppose $D\ge3$.
If $x$ and $y$ are vertices in $\gex$ such that $d(x,y)\ge3D$,
then there exists a vertex $z$ in $\gex$ which is at least of distance $D$ from $x$ and $y$, and whose star separates $x$ from $y$ in $\gex$.
\end{enumerate}
\end{lem}

\begin{proof}
(1)
By the proof of Lemma~\ref{l:extension}, there exist $\gam_0\subseteq \gam_1\subseteq\cdots\gam_k\subseteq\gex$ such that $\gam_0$ is a conjugate of $\gam$ containing $x$, $\gam_k$ contains $y$, and $\gam_{i+1}$ is obtained from $\gam_i$ by doubling along the star of a vertex $v_i$ in $\gam_i$, for $i=0,\ldots,k-1$.
Let us choose $k$ to be minimal, so that $\st(v_k)$ separates $x$ from $y$ in $\gex$. 
We may assume $\gam_0=\gam$.
We claim that there exist $l\ge0$,
$g_0=1,g_1,\ldots,g_l\in A(\gam)$, $x_i\in V(\gam^{g_{i-1}})\cap V(\gam^{g_i})$ for $i=1,\ldots,l$,
and  $\Lambda_0\subseteq\Lambda_1\subseteq\cdots\Lambda_l\subseteq\gex$ such that
\begin{enumerate}[(i)]
\item
$x\in \Lambda_0=\gam$ and $y\in\gam^{g_l}\subseteq\Lambda_l\subseteq\gam_k$;
\item
$\Lambda_i=\Lambda_{i-1}\cup\gam^{g_i}$ and $\Lambda_{i-1}\cap\gam^{g_i}\subseteq\st(x_i)$ for $i=1,\ldots,l$;
\item
$\st(x_i)$ separates $x$ from $y$ in $\gex$  for $i=1,\ldots,l$.
\end{enumerate}
If $k=0$, then $l=0$ and so the claim is clear. To use an induction, we assume the claim for $k-1$.
Write $\gam_k = \gam_{k-1}\cup\gam_{k-1}^z$ for some $z\in V(\gam_{k-1})$. 
By inductive hypothesis, one can construct
$x\in \Lambda_0\subseteq\ldots\subseteq\Lambda_m\subseteq\gam_{k-1}$ and $z\in \Lambda'_0\subseteq\ldots\subseteq\Lambda'_{m'}\subseteq\gam_{k-1}^z$ such that $z\in\gam^{g_m}\subseteq\Lambda_m$, $y\in \Lambda'_{m'}$, and the conditions (i),(ii),(iii) above are satisfied. 
Let us define $\Lambda_{m+i} = \Lambda_m\cup\Lambda'_{i-1}$ for $i=1,\ldots,m'+1$. 
This means, in particular, that $x_{m+1}=z$ and $\gam^{g_{m+1}}=\Lambda'_0$. 
Note that $\Lambda_{m+m'+1}\subseteq\gam_k$ and $\st(x_{m+1})=\st(z)$ separates $x$ from $y$ in $\gex$. If $\st(x_i)$ did not separate $x$ from $y$ in $\gex$  for $i\ne m+1$, one would get a contradiction by finding a path from $x$ to $z$ and $z$ to $y$ avoiding $\st(x_i)$. The claim is proved.

(2) In (1), we note that $d(x_i,x_{i+1})\le D$ for each $i=0,\ldots,l$.
There exists $j$ such that $d(x_0,x_{j-1})\le D$ and $d(x_0,x_j)\ge D$.
Then $d(x_{j-1},x_l)\ge 2D$ and $d(x_j,x_l)\ge d(x_{j-1},x_l) - d(x_{j-1},x_j) \ge D$.
\end{proof}


\begin{lem}\label{l:disjoint}
Suppose $\gam$ is a finite graph with at least two vertices such that $\gam$ does not split as a nontrivial join.
If $\Lambda_1\le\gex$ and $\Lambda_2\le\gex$, then $\Lambda_1\coprod\Lambda_2\le\gex$.
\end{lem}

\begin{proof}
We may only consider the case $\Lambda_1=\Lambda=\Lambda_2$; 
for in general, we can just take $\Lambda=\Lambda_1\cup\Lambda_2$.
We may further assume that $\Lambda$ does not split as a nontrivial join; 
otherwise, replace $\Lambda$ by $\Lambda\cup\gam$. 
Write the vertices of $\Lambda$ as $v_1^{w_1}, v_2^{w_2}, \ldots, v_r^{w_r}$ where $v_i\in V(\gam), w_i\in A(\gam)$. 
Following the notation in the proof of Lemma~\ref{l:extension} (4),
consider an embedding $\phi: \aga\to \Mod(\Sigma)$ for some closed surface $\Sigma$ such that each vertex $v$ of $\gam$ is mapped to some power of the Dehn twist about a simple closed curve $\alpha(v)$. 
We take the union of the regular neighborhoods of the curves $\phi(w_i^{-1}).\alpha(v_i)$ for $i=1,\ldots,r$ and cap off the null-homotopic boundary curves, to get a connected subsurface $\Sigma_0$ of $\Sigma$. 
Let $\psi$ be some product of powers of the Dehn twists about $\phi(w_i^{-1}).\alpha(v_i)$ for $i=1,\ldots,r$ to get a pseudo-Anosov homeomorphism on $\Sigma_0$, as per Lemma \ref{l:filling curves}.  There exists $M>0$ such that 
$\psi^M\phi(w_i^{-1}).\alpha(v_i)$ and $\phi(w_j^{-1}).\alpha(v_j)$ essentially intersect for any $i,j=1,\ldots,r$. Note that $\psi^{-M}$ is the image of some word $w\in\aga$ by the embedding $\phi$. It follows that an arbitrary vertex $v_i^{w_iw}$ of $\Lambda.w$ is neither equal nor adjacent to any vertex $v_j^{w_j}$ of $\Lambda$.
\end{proof}

Let $\gam$ be a graph. We say \emph{$(v_1,\ldots,v_n)$ spans an induced $P_n$} for $v_1,\ldots,v_n \in V(\gam)$ if $\{v_1,\ldots,v_n\}$ induces $P_n$ in $\gam$ and $v_i$ and $v_{i+1}$ are adjacent for $i=1,\ldots,n-1$.
We also say \emph{$(v_1,\ldots,v_n)$ spans an induced $C_n$},
if $\{v_1,\ldots,v_n\}$ induces $C_n$ in $\gam$ and $v_i$ and $v_{i+1}$ are adjacent for $i=1,\ldots,n$ with convention that $v_{n+1}=v_1$.

\begin{lem}\label{l:star double detail}
Suppose $t$ is a vertex of a finite graph $\gam$. 
Let $\gam^*$ denotes the double of $\gam$ along the star of $t$.
\begin{enumerate}
\item
If $\gam^*$ contains an induced $C_n$ for some $n\ge6$, then $\gam$ contains an induced $C_m$ for some $5\le m\le n$.
\item
If $\gam^*$ contains an induced $C_n\opp$ for some $n\ge5$, then $\gam$ contains an induced $ C_n\opp$ or $C_{n+1}\opp$.
\item
If $\gam^*$ contains an induced $P_4$, then $\gam$ contains an induced $P_4$.
\end{enumerate}
\end{lem}

\begin{proof}
Let $L$ be the link of $t$ in $\gam$ and $A = V(\gam)\smallsetminus (L\cup\{t\})$.
Take an isomorphic copy $\gam'$ of $\gam$, and let $A'$ be the image of $A$ in $\gam'$.
We may write $\gam^* = \gam\cup_\sigma\gam'$, where
 $\sigma$ is the restriction on $L\cup\{t\}$ of the given isomorphism between $\gam$ and $\gam'$.
The image of $L$ or $t$ in $\gam^*$ is still denoted by $L$ or $t$, respectively.
Let $\mu : \gam^*\to \gam$ be the natural retraction so that $\mu(A')=A$.
Note $V(\gam^*) = A\cup A'\cup L\cup \{t\}$, and $L\cup\{t\}$ separates $\gam^*$ into induced subgraphs on $A$ and on $A'$; see Figure~\ref{f:stardouble} (a).

(1) Suppose $\Omega\cong C_n$ is an induced subgraph of $\gam^*$ for some $n\ge 6$,
and assume the contrary of the conclusion.
If $t$ is in $\Omega$, then $V(\Omega)\smallsetminus \st_\Omega(t) = V(\Omega)\smallsetminus (L\cup\{t\})$ 
induces a connected graph in $\gam^*$  
and hence, $V(\Omega)$ is contained either in $A\cup L\cup\{t\}$
or $A'\cup L\cup\{t\}$; in particular, $\gam$ would contain an induced $C_n$. 
 
So we have $t\not\in\Omega$. 
Suppose the tuple of vertices $(v_1,\ldots,v_n)$ spans $\Omega\cong C_n$. 
We will take indices of $v_i$ modulo $n$.
If $\{v_{i+1},\ldots,v_{i+k}\}$ form a maximal path that is contained in $A$ or in $A'$ for some $1\le k\le n-3$,
then $v_i,v_{i+k+1}\in L$ and hence,
 $(t,v_i,\mu(v_{i+1}),\ldots,\mu(v_{i+k}),v_{i+k+1})$ spans an induced $C_{k+3}$ in $\gam$; 
hence, $k = 1$.
This means that
if $v_i$ is in $A$ or $A'$, then $\lk_\Omega(v_i)$ is contained in $L$.
If $v_i,v_j\in A\cup A'$ for $i\ne j$ and $\mu(v_i)=\mu(v_j)$, then $v_{j\pm1}\in L$ are adjacent to $v_i$ and hence, $C_n\cong C_4$; 
this implies that $\mu(v_1),\ldots,\mu(v_n)$ are all distinct.
Since $\mu(\Omega)$ is not an induced $C_n$ in $\gam$,
$\mu(v_i)$ and $\mu(v_j)$ are adjacent for some $v_i\in A$ and $v_j\in A'$.

\textit{Case 1.}
	Suppose some $x\in \lk_\Omega(v_i)\smallsetminus \lk_\Omega(v_j)$ 
	is non-adjacent to some $y\in\lk_\Omega(v_j)\smallsetminus \lk_\Omega(v_i)$.
	Then, $(t,x,v_i,\mu(v_j),y)$ spans an induced $C_5$ in $\gam$.

\textit{Case 2.}
Suppose Case 1 does not occur. 
This implies that $\lk_\Omega(v_i)$ and $\lk_\Omega(v_j)$ are neither equal nor disjoint.
One can write $\lk_\Omega(v_i)=\{x,y\}$ and $\lk_\Omega(v_j)=\{x,z\}$ such that $y$ and $z$ are adjacent.
Then $(v_i,x,v_j,z,y)$ spans an induced $C_5$ in $\Omega$, which is a contradiction.

(2)
Suppose $\Omega\cong C_n$ is an induced subgraph of $(\gam^*)\opp$ for some $n\ge5$,
and assume the contrary of the conclusion.
Let $\Lambda\le\gam\opp$ and $\Lambda'\le(\gam')\opp$ be the induced subgraphs on $A$ and on $A'$, respectively.
Note that ${\Lambda}\ast{\Lambda'}\le ({\gam^*})\opp$; see Figure~\ref{f:stardouble} (b).

First suppose $t\in\Omega$ and write $\lk_\Omega(t)=\{a,a'\}$.
If $a,a'\in A$, then $V(\Omega)\subseteq \{t\}\cup A\cup L$ and hence, $\Omega\le{\gam}\opp$.
Similarly, it is not allowed that $a,a'\in A'$. 
Hence, we may assume $a\in A$ and $a'\in A'$;
this would still be a contradiction since $a$ and $a'$ are adjacent in $({\gam'})\opp$.

This shows $t\not\in\Omega$.
Note that $\Omega\cap (\Lambda\ast\Lambda')$ has at most three vertices, 
since so does every non-trivial join subgraph of $\Omega\cong C_n$.

\textit{Case 1.}
	Suppose $V(\Omega\cap\Lambda)=\{a\}$ and $V(\Omega\cap\Lambda')=\{a'\}$ for some $a,a'$.
	We label cyclically $V(\Omega)= \{a,a',v_1,\ldots,v_{n-2}\}$ where $v_1,\ldots,v_{n-2}\in L$.
 	Since $\Omega$ is not a triangle, $\mu(a')\ne a$.
    If $\mu(a')$ is adjacent to $a$ in $\Lambda$, $\mu(\Omega)$ is an induced $C_n$ in ${\gam}\opp$. If $\mu(a')$ is not adjacent to $a$ in $\Lambda$, then $(a,t,\mu(a'),v_1,\ldots,v_{n-2})$ spans an induced $C_{n+1}$ in $\gam\opp$.
	
\textit{Case 2.}
    $V(\Omega\cap\Lambda)=\{a\}$ and $V(\Omega\cap\Lambda')=\{a',a''\}$ for some $a,a',a''$:
	write $V(\Omega)= \{a'',a,a',v_1,\ldots,v_{n-3}\}$ where $v_1,\ldots,v_{n-3}\in L$.
	Then $(\mu(a''),t,\mu(a'),v_1,\ldots,v_{n-3})$ spans an induced $C_n$ in ${\gam}\opp$.
	
(3) Suppose $\Omega\cong P_4$ is an induced subgraph of $\gam^*$, and assume $\gam$ is $P_4$--free.
Since $\Omega$ intersects both $A$ and $A'$, we have $|V(\Omega)\cap (L\cup\{t\})|\le 2$. If $t\in\Omega$, then $| \Omega\cap L| = 1$; this is a contradiction, for $\Omega\cap (L\cup \{t\})$ separates $\Omega$ while the valence of $t$ in $\Omega$ is $1$. Hence $t\not\in\Omega$. 

Now if $\Omega\cap(L\cup\{t\})=\Omega\cap L$ is a single vertex, then one of $A$ or $A'$ intersect $\Omega$ at two verties, and those two vertices along with $t$ and $\Omega\cap L$ span an induced $P_4$ in $\gam$.
If $\Omega\cap L$ has two vertices, say $x$ and $x'$, then $\Omega\cap A=\{a\},\Omega\cap A'=\{a'\}$ for some vertices $a$ and $a'$. We note that $a$ and $\mu(a')$ are adjacent, since $\gam$ is $P_4$--free.
Without loss of generality, we may assume $(a,x,x',a')$ or  $(x,a,x',a')$ spans $\Omega\cong P_4$; in either case, $(t,x,a,\mu(a'))$ spans an induced $P_4$ in $\gam$.
\end{proof}

\begin{figure}[htb!]
  \tikzstyle {a}=[red,postaction=decorate,decoration={%
    markings,%
    mark=at position 1 with {\arrow[red]{stealth};}}]
  \tikzstyle {b}=[blue,postaction=decorate,decoration={%
    markings,%
    mark=at position .43 with {\arrow[blue]{stealth};},%
    mark=at position .57 with {\arrow[blue]{stealth};}}]
  \tikzstyle {v}=[draw,shape=circle,inner sep=0pt]
  \tikzstyle {bv}=[black,draw,shape=circle,fill=black,inner sep=1pt]
  \tikzstyle{every edge}=[-,draw]
\subfloat[(a) $\gam^*$]{
	\begin{tikzpicture}[thick]
\node [circle,draw,inner sep=10pt] at (0,0) (l) {};
\node [circle,draw,inner sep=15pt] at (-1.8,0) (a) {};
\node [circle,draw,inner sep=15pt] at (1.8,0) (ap) {};
\draw (1.8, 1) node [above] {$A'$};
\draw (-1.8, 1) node [above] {$A$};
\draw (0, 1) node [above] {$L$};
\draw (1.8,.35) node[bv] (apu) {};
\draw (1.8,.-.35) node[bv] (apl) {};
\draw (-1.8,.35) node[bv] (au) {};
\draw (-1.8,.-.35) node[bv] (al) {};
\draw (0,.25) node[bv] (lu) {};
\draw (0.25,.-.2) node[bv] (ll) {};
\draw (apl) -- (lu) -- (al);
\draw (0,-1.5) node [bv] (t) {} node [below] {$t$};
\draw (lu) -- (t) -- (ll);
\end{tikzpicture}
}
\hspace{.4in}
\subfloat[(b) $(\gam^*)\opp$]{
	\begin{tikzpicture}[thick]
\node [circle,draw,inner sep=10pt] at (0,1) (l) {};
\node [circle,draw,inner sep=15pt] at (-1.8,-.2) (a) {};
\node [circle,draw,inner sep=15pt] at (1.8,-.2) (ap) {};
\draw (2.4, 1) node [] {$A'$};
\draw (-2.4,1) node [] {$A$};
\draw (.6, 1.2) node [right] {$L$};
\draw (1.8,.15) node[bv] (apu) {};
\draw (1.8,.-.55) node[bv] (apl) {};
\draw (-1.8,.15) node[bv] (au) {};
\draw (-1.8,.-.55) node[bv] (al) {};
\draw (0,1.2) node[bv] (lu) {};
\draw (0,.75) node[bv] (ll) {};
\draw (apu) -- (lu) -- (au);
\draw (ll) -- (lu);
\draw (apu) -- (t) -- (apl);
\draw (au) -- (t) -- (al);
\draw (au) -- (apl) -- (al) -- (apu) -- (au);
\draw (0,-1.5) node [bv] (t) {} node [below] {$t$};
\end{tikzpicture}}
  \caption{Proof of Lemma~\ref{l:star double detail}.}
  \label{f:stardouble}
\end{figure}
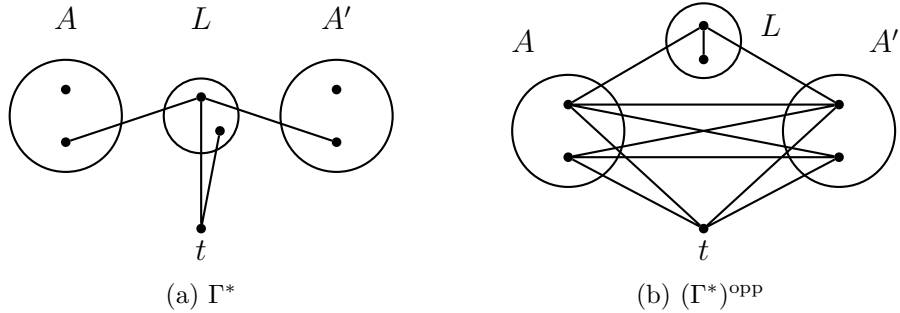

The extension graph of a given graph preserves $C_m$--freeness for certain $m$, as described in Lemma~\ref{l:cmfree}.
\begin{lem}\label{l:cmfree}
Suppose $\gam$ is a finite graph.
\begin{enumerate}
\item
If $\gam$ is triangle--free, then so is $\gex$.
\item
If $\gam$ is square--free, then so is $\gex$.
\item
Suppose $n\ge5$, and $\gam$ is triangle--free or square--free.
If $\gam$ is $C_m$--free for every $m=5,\ldots,n$, then so is $\gex$.
\item
If $\gam$ is weakly chordal, then so is $\gex$.
\item
If $\gam$ is bipartite, then so is $\gex$.
\end{enumerate}
\end{lem}

\begin{proof}
(1) is immediate from the definition of $\gex$. (2) is similar to, and much easier than, the proof of Lemma~\ref{l:star double detail}; a key observation is that if $\gam$ is square--free and $t$ is a vertex, then $\gam\smallsetminus\{t\}$ does not have an induced $P_3$ which intersects $\lk(t)$ only on its endpoints. The proofs of (3) and (4) are direct consequences of Lemma~\ref{l:star double detail} and the fact that $C_n\opp$ contains a triangle and a square for every $n\ge6$.
For (5), consider the pullback of a $2$--coloring of $\gam$ by the retraction $\gex\to\gam$.
\end{proof}

\begin{rem}
The class $\mathcal{W}$ of weakly chordal graphs is closed under edge contractions~\cite[Theorem 4.7]{Kim2007}. 
Oum pointed to us that $\mathcal{W}$ is closed under taking the double along the \textit{link} of a vertex (private communication). 
Lemma~\ref{l:cmfree} (4) follows the lines of these two results.
\end{rem}

\begin{lem}\label{l:thin bigon}
Suppose $\gam$ is a finite weakly chordal graph. 
Then $\gex$ satisfies the \emph{2--thin bigon property}. 
Namely, suppose that $x$ and $y$ are two geodesic segments connecting two vertices $v$ and $w$.  Then $x$ is contained in a $2$--neighborhood of $y$.
\end{lem}
\begin{proof}
Lemma~\ref{l:cmfree} (4) implies that $\gex$ is weakly chordal.
Since $x$ is geodesic, non-consecutive vertices of $x$, or those of $y$, are non-adjacent.
If one vertex in $y$ is not adjacent to any vertex in $x$, it is easy to check that both of its neighbors in $y$ are adjacent to a vertex of $x$; otherwise, $\gex$ would contain an induced long cycle.
Therefore, the distance between any vertex of $y$ and a vertex of $x$ is at most two.
\end{proof}

We will need the following observation later on:

\begin{lem}\label{l:conjugate}
Let $n\geq 5$ and consider an arbitrary inclusion $i:C_n\to C_n^e$.  Then $i(C_n)$ is conjugate to the original copy of $C_n$.
\end{lem}
\begin{proof}
Suppose that the inclusion $i$ maps $C_n$ into $C_n^e$ in such a way that the image $i(C_n)$ is not contained in a conjugate of $C_n$.  
By Lemma~\ref{l:extension} (6), there is a star of a vertex $v$ which separates $i(C_n)$ into at least two smaller connected subgraphs. Let $A$ and $B$ be the closures of two distinct components of $i(C_n)\smallsetminus\st(v)$. 
Since $C_n^e$ is triangle-free and square-free, $B\not\cong P_3$. Hence, the induced subgraph of $\gex$ spanned by $A\cup\{v\}$ is a cycle of length strictly less than $n$.
This is a contradiction to Lemma~\ref{l:cmfree} (3).
\end{proof}

\section{Right-angled Artin subgroups of right-angled Artin groups}\label{sec:gex and raag}
In this section we will prove Theorems \ref{t:thm1} and \ref{t:thm2}.  To get a more concrete grip on $\gex$, one can check the following three examples.  In the case where $\gam$ is a complete graph, $\gex=\gam$.  In the case where $\gam$ is discrete and $|V|>1$ then $\gex$ is a countable union of vertices with no edges.

In the case where $\gam$ is a square, we have our first nontrivial example.  Label the vertices of $\gam$ by $\{a,b,c,d\}$, with $a$ and $c$ connected to $b$ and $d$.  Performing the construction of $\gex$, we see that the vertices consist of all $c$--conjugates of $a$, all $d$--conjugates of $b$, all $a$--conjugates of $c$ and all $b$--conjugates of $d$.  Note also that each conjugate of $a$ and $c$ is connected to each conjugate of $b$ and $d$.  It follows that $\gex$ is isomorphic to a complete bipartite graph on two countable sets.

We will give two proofs of Theorem \ref{t:thm1}.  The first will use Dehn twists and the result from \cite{Koberda2011}.  The other will use pseudo-Anosov homeomorphisms and the result from \cite{CLM2010}. Also, one can deduce Theorem~\ref{t:thm1} from Corollary~\ref{c:double}, of which we give an alternative, topological proof in Section~\ref{sec:appendix}.

\begin{proof}[First proof of Theorem \ref{t:thm1}]
Let us recall the notations from the proof of Lemma~\ref{l:extension} (4). 
There exists a closed surface $\Sigma$, an embedding $\alpha:\gam\to\mC(\Sigma)$ and $N>0$ such that  the map $\phi:A(\gam)\to\Mod(\Sigma)$ defined by $\phi(v)=T_{\alpha(v)}^N$ for each $v\in V(\gam)$ is injective. 
If we extend $\alpha$ to an embedding $\alpha^e:\gex\to\mC(\Sigma)$ by $\alpha^e(v^w)=\phi(w^{-1}).\alpha(v)$ for $v\in V(\gam)$ and $w\in A(\gam)$, then the image of $\alpha^e$ is an induced subgraph of $\mC(\Sigma)$.

Now put $V(\Lambda)=\{v_1^{w_1},\ldots,v_n^{w_n}\}\subseteq V(\gex)$, where $v_i\in V(\gam)$ and $w_i\in A(\gam)$ for $i=1,\ldots,n$. The co--incidence graph of $\{\alpha^e(v_1^{w_1}),\ldots,\alpha^e(v_n^{w_n})\}$ is $\Lambda$. By Lemma~\ref{l:mod}, there exists an $M>0$ such that the map  $\psi: A(\Lambda)\to \Mod(\Sigma)$ defined by
$\psi(v_i^{w_i}) = T^{MN}_{\alpha^e(v_i^{w_i})}$ is injective.
Since
\[
T^{MN}_{\alpha^e(v_i^{w_i})} = T^{MN}_{\phi(w_i^{-1}).\alpha(v_i)} = 
\phi(w_i^{-1})\circ T^{MN}_{\alpha(v_i)}\circ \phi(w_i) = \phi((v_i^M)^{w_i}),
\]
we have that $\psi$ factors through $\phi$ as follows.

\[
\xymatrix{
A(\Lambda)\ar@{^{(}-->}[dr]\ar@{^{(}->}[rr]^\psi && \Mod(\Sigma)\\
& A(\gam)\ar@{^{(}->}[ur]^\phi
}
\]
Since $\psi$ is injective, we have an embedding from $A(\Lambda)$ to $A(\gam)$.
\end{proof}

\begin{proof}[Second proof of Theorem \ref{t:thm1}]
Choose a closed surface $\Sigma$ with a configuration of subsurfaces and pseudo-Anosov homeomorphisms whose co--incidence graph is $\gam$ and which satisfy the technical hypothesis of \cite{CLM2010}.  The second proof of the theorem is essentially identical to the first one, with Dehn twists replaced by pseudo-Anosov homeomorphisms.  The only nuance is that for each finite subgraph of $\gex$ which we produce, we must show that we do not get any unexpected nesting of subsurfaces.

The easiest way to avoid unexpected nesting is to arrange for the pseudo-Anosov generators of $A(\gam)$ to be supported on surfaces with no inclusion relations between them.  It is clear and can be seen from \cite{Koberda2011} that one can find a configuration of subsurfaces with the desired intersection correspondence on a surface of sufficiently large genus.  Furthermore, one can arrange for these subsurfaces to all have a given genus $g$ and one boundary component.  If $\gam$ has $n$ vertices $\{v_1,\ldots,v_n\}$, modify the subsurface corresponding to $v_i$ to have genus $g+n-i$, $i$ punctures and one boundary component.  It is clear then that no two subsurfaces produced in this way can be nested in a way which sends punctures to punctures.  After finding a pseudo-Anosov homeomorphism on each of these subsurfaces which does not extend to the subsurfaces with the punctures filled in, we can apply the main result of \cite{CLM2010}.
\end{proof}

\begin{proof}[Proof of Corollary \ref{c:double}]
Note that the induced subgraph of $\gex$ on $V(\gam)\cup V(\gam).t$ is isomorphic to $D_{\st(t)}(\gam)$.
Theorem \ref{t:thm1} completes the proof.
\end{proof}

We now turn our attention to Theorem \ref{t:thm2}. 
Suppose two cyclically reduced words $w$ and $v$ have pure factor decompositions $w=w_1\cdots w_n$ and $v=v_1\cdots v_m$. Then $w$ and $v$ do not commute if and only if there is a pair of pure factors $w_i$ and $v_j$ such that $[w_i,v_j]\neq 1$, which follows easily from the Centralizer Theorem.
Using the pure factor decomposition, we can give the following result concerning the structure of copies of $\bZ*\bZ^2$ in a right-angled Artin group:

\begin{lem}
Suppose cyclically reduced words $a,b$ and $x$ generate a copy of $\bZ*\bZ^2$ of $A(\gam)$, where the splitting is given by \[\langle x\rangle*\langle a,b\rangle.\]  Let $\{x_1\cdots x_k\}$ be the sets of pure factors of $x$.
Then there is a pure factor $x_l$ which commutes with neither $a$ nor $b$.
\end{lem}
\begin{proof}
Suppose to the contrary that every pure factor appearing in the pure factor decomposition of $x$ commutes with either $a$ or $b$.  We may assume that \[x=(x_1^{e_1}\cdots x_m^{e_m})\cdot x_{m+1}^{e_{m+1}}\cdots x_n^{e_n},\] where each of $\{x_1,\ldots,x_m\}$ commute with $a$ and each of $\{x_{m+1}\cdots x_n\}$ commute with $b$ for some integers $e_1,\ldots,e_n$.  We then form the commutator $[b,a^x]$, which is nontrivial in $\bZ*\bZ^2$.  Note that this is just the commutator of $b$ with $a$ conjugated by $x_{m+1}^{e_{m+1}}\cdots x_n^{e_n}$.  Since this last element commutes with $b$, and $a$ commutes with $b$, the commutator is trivial; so we have a contradiction.
\end{proof}

Let $\gam$ be a finite graph.
For $W\subseteq A(\gam)$, the \emph{commutation graph} of $W$ is a graph with vertex set $W$ such that two vertices are adjacent if and only if they commute in $A(\gam)$. The following is a key step in the proof of Theorem~\ref{t:thm2}.

\begin{lem}\label{l:joingraph2}
Let $\gam$ be a finite graph and $W=\{w_1,\ldots,w_n\}$ be a set of conjugates of pure factors in $A(\gam)$ such that $w_i\ne w_j^{\pm1}$ for any $i\ne j$.
Then the commutation graph of $W$ embeds into $\gex$ as an induced subgraph.
\end{lem}
\begin{proof}
Let $Z$ be the commutation graph of $W$, and write $w_i = u_i^{p_i}$ where $u_i$ is a pure factor and $p_i\in A(\gam)$. 
Consider an embedding $\phi:A(\gam)\to \Mod(\Sigma)$ for some closed surface $\Sigma$, as described in Lemma~\ref{l:mod}; here, each vertex $v$ maps to a power of a Dehn twist along a simple closed curve $\alpha(v)$. We set $\Sigma_i$ to be the regular neighborhood of the union of the simple closed curves $\{\alpha(v):v\in \supp(u_i)\}$; we will cap off any null-homotopic boundary components of $\Sigma_i$. Since $u_i$ is a pure factor, $\Sigma_i$ is connected. 
Centralizer Theorem implies that the co--incidence graph of $\{\phi(p_i)^{-1}(\Sigma_i)\}$ is the same as $Z$. 
Note our convention that if
$\phi(p_i)^{-1}(\Sigma_i)$ and $\phi(p_j)^{-1}(\Sigma_j)$ are isotopic and 
$i\ne j$, then the two corresponding vertices in the co--incidence graph are declared to be distinct and non-adjacent.

A certain multiplication of Dehn twists along $\{\alpha(v):v\in \supp(u_i)\}$ will give a pseudo-Anosov homeomorphism $\psi_i$ on $\Sigma_i$ by Lemma~\ref{l:filling curves}.
One may choose $\psi_i = \phi(q_i)$ for some $q_i\in\form{\supp(u_i)}$.
For each $i$, arbitrarily fix $s_i\in\supp(u_i)$.
Then for some sufficiently large $M$,
the co--incidence graph $Z'$ of $\{\phi(p_i)^{-1}(\psi_i^{-M}(\alpha(s_i))):i=1,2,\ldots,n\}$ is the same as the co-incidence graph of 
$\{\phi(p_i)^{-1}(\Sigma_i)\}$.
Note that $Z'$ is the same as the induced subgraph of $\gex$ spanned by $\{s_i^{q_i^M p_i}:i=1,\ldots,n\}$.
\end{proof}

Therefore, to prove that a particular graph is an induced subgraph of $\gex$, it suffices to exhibit it as the commutation graph of some conjugates of pure factors.  We remark that Lemma \ref{l:joingraph2} can be proved alternatively using the primary result of \cite{CLM2010}.  The proof carries over nearly verbatim, replacing Dehn twists and annuli with pseudo-Anosov homeomorphisms and the surfaces on which they are supported.

Theorem~\ref{t:thm2} is an easy consequence of the following.

\begin{thm}\label{t:thm2 strong}
Suppose $\Lambda$ and $\gam$ are finite graphs and $A(\Lambda)\le \aga$.
\begin{enumerate}[(1)]
\item
There is an embedding $\phi:A(\Lambda)\to \aga$ that factors through $\psi:A(\Lambda)\to A(\gex)$ and the natural retraction $\pi:A(\gex)\to A(\gam)$ such that for each vertex $v$ of $\Lambda$, $\supp(\psi(v))$ induces a clique $K_v\le\gex$.
\item
In (1), $\psi$ can be chosen so that $K_v\not\subseteq K_w$ for any $v\ne w\in V(\Lambda)$.
\end{enumerate}
\end{thm}

\begin{proof}
(1)
We write $V(\Lambda) = \{v_1,\ldots,v_n\}$.
Suppose $\iota:A(\Lambda)\to A(\gam)$ is an inclusion, and put $\iota(v_i)=w_i$.
For each $i=1,\ldots,n$, we have a unique pure factor decomposition
\[
w_i = \left(\prod_{j=1}^{n(i)}u_{i,j}^{a_{i,j}}\right)^{p_i}
\]
where $a_{i,j}\ne0$, $p_i\in A(\gam)$ and $u_{i,j}$ is a pure factor. We can further assume that no two elements in $\{u_{i,j}:i=1,\ldots,n\mbox{ and }j=1,\ldots,n(i)\}$ are inverses of each other.
Let $Z$ be the commutation graph of $\{u_{i,j}^{p_i}\}$.
There is a clique in $Z$ corresponding to the words $\{u_{i,j}^{p_i}:j=1,\ldots,n(i)\}$ for each $i$.

We denote the vertex of $Z$ corresponding to $u_{i,j}^{p_i}$ by $k_{i,j}$.
Let $\rho:A(Z)\to A(\gam)$ be the map sending $k_{i,j}$ to $u_{i,j}^{p_i}\in A(\gam)$. 
There is a map $\lambda:A(\Lambda)\to A(Z)$ which sends each vertex $v_i$ to the element \[\prod_{j=1}^{n(i)} k_{i,j}^{a_{i,j}}\]
Notice that $\iota=\rho\circ\lambda$ and so, $\lambda$ is injective.
By Lemma \ref{l:joingraph2}, the graph $Z$ embeds in $\gex$ as an induced subgraph.
By the proof of Theorem~\ref{t:thm1}, there is an embedding $\rho':A(Z)\to A(\gex)$ where $\rho'$ maps each vertex $k_{i,j}$ in $Z$ to a power of a vertex in $\gex$, such that $\pi\circ\rho'$ is injective.
We define $K_i$ to be the clique of $\gex$ induced by $\{\supp(\rho'(k_{i,j})):j=1,\ldots,n(i)\}$ and $\psi=\rho'\circ\lambda$. Note $\pi\circ\psi=(\pi\circ\rho')\circ\lambda$ is injective; see the following commutative diagram.
\[
\xymatrix{
&&A(\gex)\ar[dr]^\pi\\
A(\Lambda)\ar@/_1pc/[rrrd]_{\iota}\ar@/^1pc/[rru]^\psi\ar@{^{(}->}[r]_\lambda& A(Z)\ar@{^{(}->}[ur]_>>>>{\rho'}\ar@{^{(}->}[rr]
\ar[rrd]_\rho &&A(\gam) \\
 & &&A(\gam)
}
\]

For each vertex $v_i$ of $\Lambda$, 
\[
\psi(v_i) = \rho'\left(\prod_{j=1}^{n(i)} k_{i,j}^{a_{i,j}}\right)
\]
is generated by the vertices of $K_i$.

(2) Choose $\psi$ in (1) such that $\sum_{v\in V(\Lambda)} |\supp(\psi(v))|$ is minimal. Assume the contrary so that for some $v\ne w\in V(\Lambda)$,
\[
\psi(v) = a_1^{e_1}\cdots a_r^{e_r}\quad\textrm{ and }\quad
\psi(w) = a_1^{f_1}\cdots a_r^{f_r}
\]
where $\{a_1,\ldots,a_r\}$ induces a clique in $\gex$, $e_1\ne0$,
and $f_i\ne 0$ for each $i=1,\ldots,r$.
There exists $\ell m\ne0$ such that $\psi(v^{\ell} w^m)=a_2^{d_2}\cdots a_r^{d_r}$
for some $d_2,\ldots,d_r\in\Z$.
We note that if $[w,x]=1$ for some $x\in V(\Lambda)$, then each $a_i$ is adjacent to each vertex in $\supp(\psi(x))$, and so, $[\psi(v),\psi(x)]=1$.
This shows $\lk(w)\subseteq\st(v)$.
Define $\mu,\mu':A(\Lambda)\to A(\Lambda)$ by $\mu(w)=vw, \mu'(w)=w^m$ and $\mu(x)=x=\mu'(x)$ for $x\in V(\gam)\smallsetminus\{w\}$. Then $\mu'$ is a monomorphism and $\mu$ is an isomorphism, called a \emph{transvection}~\cite{Servatius1989}.

Define a new embedding $A(\Lambda)\to A(\gex)$ by $\psi' = \psi\circ\mu'\circ\mu^{\ell}$.
For a vertex $x$, we have
\[
\psi'(x)
= \begin{cases}
\psi(x)\quad & \textrm{if }x \ne w,\\
a_2^{d_2}\cdots a_r^{d_r}& \textrm{if }x=w.
\end{cases}
\]
Since $|\supp(\psi'(w))|\le r-1 < |\supp(\psi(w))|$, we have a contradiction.
\end{proof}

\begin{cor}[cf. \cite{Kambites2009}]\label{c:wkambites}
Let $\gam$ be a finite graph and let $F_2\times F_2<A(\gam)$.  Then $\gex$ contains an induced square.
\end{cor}

This is not the exact statement of Kambites' Theorem (Corollary \ref{c:kambites}), but it is the most important step in the proof.  Combining Corollary~\ref{c:wkambites} and Lemma~\ref{l:cmfree} (2), we obtain the statement given by Kambites.

\begin{proof}[Proof of Corollary \ref{c:wkambites}]
Label the edges of a square cyclically by \[\{a,b,c,d\}.\]  By Theorem \ref{t:thm2}, there exists an embedding of the square $C_4\to \gex_k$ so that the support of each edge of $C_4$ in $\gex$ is a clique.  Denote the supports of the vertices of $C_4$ in $\gex$ by $\{V_a,V_b,V_c,V_d\}$.  There exist nonadjacent vertices $x$ and $y$ in $V_a$ and $V_c$ respectively, and nonadjacent vertices $w$ and $z$ in $V_b$ and $V_d$ respectively.  The vertices $x$ and $y$ are clearly distinct and are adjacent to each vertex of $V_b$ and $V_d$ and are therefore distinct from $w$ and $z$.  
It follows that $(x,w,y,z)$ is an induced cycle of length four in $\gex$.
\end{proof}

\section{Trees and the path on four vertices}\label{sec:trees}
Recall that a \textit{forest} is a disjoint union of trees.
In this section, we characterize the right-angled Artin groups that contain, or are contained in, the right-angled Artin groups on forests. 

\begin{prop}\label{p:only forest}
Let $\gam, \gam'$ be finite graphs such that $\gam$ is a forest.
If $A(\gam')$ embeds into $A(\gam)$, then $\gam'$ is also a forest.
\end{prop}
\begin{proof}
Since $A(\gam)$ is a three-manifold group~\cite{Brunner1992}, 
so is $A(\gam')$ and $\gam'$ is a disjoint union of trees and triangles~\cite{Droms1987}. 
As the maximum rank of an abelian subgroup of $A(\gam)$ is two, 
$\gam'$ does not contain a  triangle.
\end{proof}

\begin{prop}\label{p:every forest}
Every finite forest is an induced subgraph of $P_4^e$.
\end{prop}

\begin{proof}
Label $V(P_4)$ as $\{a,b,c,d\}$ such that
$[a,b]=[b,c]=[c,d]=1$ in $G=A(P_4)$.
Put $B = \form{\st(b)}=\form{a,b,c}$ and $C=\form{\st(c)}=\form{b,c,d}$.
Let $X$ be the induced subgraph of $P_4^e$
spanned by the vertices which are conjugates of $b$ or $c$.

By Lemma~\ref{l:disjoint}, it suffices to show that every finite tree $T$ is an induced subgraph of $P_4^e$.
We use an induction on the number of vertices of $T$.
Assume every tree with at most $k$ vertices embeds into $X$ as an induced subgraph,
and fix a tree $T$ with $k+1$ vertices.
Choose a valence-one vertex $v_0$ in $T$ and let $T'$ be the induced subgraph of $T$ spanned by $V(T)\smallsetminus\{v_0\}$.
By inductive hypothesis, $T'$ can be regarded as an induced subgraph of $X$. Let $v_1$ be the unique vertex of $T'$ that is adjacent to $v_0$.
Without loss of generality, we may assume that $v_1 = b$.
Write $V(T') = \{b^{u_0},b^{u_1},\ldots,b^{u_p},c^{v_1},\ldots,c^{v_q}\}$ 
such that $u_0=1$ and $u_i\not\in B$ for $i>0$.
Recall that for $i=1,2,\ldots,p$, $c^w$ and $b^{u_i}$ are non-adjacent if and only if $Bu_i\cap Cw=\varnothing$; this is equivalent to that $w\not\in CBu_i$.
Also, $c^w$ and $c^{v_j}$ are distinct if and only if $w\not\in Cv_j$.
Hence, the following claim implies that the induced subgraph of $X$ spanned by $V(T')\cup\{c^w\}$ is isomorphic to $T$ for $w = a^M$.

\textbf{Claim.} For some $M>0$, $a^M\in CB\smallsetminus(\bigcup_{i=1}^p CBu_i\cup \bigcup_{j=1}^q Cv_j)$.
Choose $M$ to be larger than the maximum number of occurrences of $a^{\pm1}$ in $u_1,\ldots,u_p,v_1,\ldots,v_q$.
It is clear that $a^M$ does not belong to $Cv_j$, 
since $a^{\pm1}$ occurs at most $M-1$ times in each element of $Cv_j = \form{b,c,d}v_j$.
Suppose $a^M \in CBu_i$ for some $i=1,\ldots,p$.
One can choose $w_1(b,d)\in\form{b,d}$ and $w_2(a,c)\in\form{a,c}$ such that $a^M = w_1(b,d)w_2(a,c)u_i$.
Since $u_i$ is not in $B$, there exists a $d^{\pm1}$ in $u_i$.
In particular, one can write $a^M = w' d^{\pm1}w'' d^{\mp1}w'''$ where $w'$ is a subword of $w_1(b,d)$, $w'''$ is a subword of $u_i$, and $w''$ is a word in $\st(d) = \form{c,d}$. The number of the occurrences of $a^{\pm1}$ on the right-hand-side is at most that of $w'''$, which is less than $M$. This is a contradiction.
\end{proof}

\begin{lem}[cf.~\cite{CLB1981}]\label{l:cograph}
For a finite graph $\gam$, the following are equivalent.
\begin{enumerate}[(i)]
\item $\gam$ is $P_4$--free.
\item $\gex$ is $P_4$--free.
\item Each connected component of an arbitrary induced subgraph of $\gam$ either is an isolated vertex or splits as a nontrivial join.
\item $\gam$ can be constructed from isolated vertices by taking successive disjoint unions and joins.
\end{enumerate}
\end{lem}
Also known as \emph{cographs}, $P_4$--free graphs are extensively studied~\cite{CLB1981}. Here, we give a self-contained proof for readers' convenience.

\begin{proof}[Proof of Lemma~\ref{l:cograph}]
(i) $\Leftrightarrow$ (ii) follows from Lemma~\ref{l:star double detail} (3).

For (ii) $\Rightarrow$ (iii), suppose $\gex$ is $P_4$--free, and choose a connected component $\Lambda$ of some induced subgraph of $\gam$. Then $\Lambda$ and $\Lambda^e$ are also $P_4$--free.
Since $\Lambda$ has no path on four vertices, it must certainly have bounded diameter. By Lemma~\ref{l:extension} (5),
This can happen only if $\Lambda$ is an isolated vertex or splits as a nontrivial join.  

(iii) $\Rightarrow$ (iv) is immediate from induction on $|V(\gam)|$.

(iv) $\Rightarrow$ (i) follows from the observation that the join of two $P_4$--free graphs are still $P_4$--free.
\end{proof}

We will now prove Theorem \ref{t:p4embed}.  Recall the statement:
\begin{thm6}
There is an embedding $A(P_4)\to A(\gam)$ if and only if $P_4$ arises as an induced subgraph of $\gam$.
\end{thm6}


\begin{proof}[Proof of Theorem~\ref{t:p4embed}]
We use an induction on the number of vertices of $\gam$.
Suppose there is an embedding $\phi:A(P_4)\to A(\gam)$ and $\gam$ is $P_4$--free.
Note that $A(P_4)$ is freely indecomposable. 
By Kurosh subgroup theorem, we may assume that $\gam$ is connected.
From the characterization of $P_4$--free graphs, one can write $\gam=\gam_1\ast\gam_2$ for some nonempty $P_4$--free graphs $\gam_1$ and $\gam_2$. 
Let $\pi_i$ denote the projection $A(\gam)\to A(\gam_i)$.  
By the inductive hypothesis, the kernel $K_i$ of $\pi_i\circ\phi$ is nontrivial. 
The subgroup $K_1K_2$ of $A(P_4)$ is isomorphic to $K_1\times K_2$.  If we can show that $K_1$ and $K_2$ are both nonabelian then we are done.  Indeed, then we obtain an embedding $A(C_4)\cong F_2\times F_2\to A(P_4)$ and this contradicts to Proposition~\ref{p:only forest}.

To see that $K_1$ and $K_2$ are both nonabelian, notice that they are both normal in $A(P_4)$. 
Fix $i=1$ or $2$.
Since $A(P_4)$ is centerless, there exist $g\in K_i$ and $h\in A(P_4)$ such that $[g,h]\ne1$.
By~\cite{Baudisch1981}, $g$ and $h$ generate a copy of $F_2$.
In particular, $g$ and $g^h$ generate a copy of $F_2$ in $K_i$.
The conclusion follows.
\end{proof}

The argument given in Theorem \ref{t:p4embed} is a reflection of a more general principle concerning right-angled Artin subgroups of right-angled Artin groups on joins:
\begin{thm}\label{t:join}
Let $\Lambda$ be a finite graph whose associated right-angled Artin group has no center and let $J=J_1*J_2$ be a nontrivial join.  Suppose we have an embedding $A(\Lambda)\to A(J)$.  Let $\pi_1$ and $\pi_2$ be the projections of $A(J)$ onto $A(J_1)$ and $A(J_2)$.  Restricting each $\pi_i$ to $A(\Lambda)$, we write $K_1$ and $K_2$ for the two kernels.  Then either at least one of $K_1$ and $K_2$ is trivial or $\Lambda$ contains a induced  square.
\end{thm}
\begin{proof}
Since $A(\Lambda)$ is embedded in $A(J)$, the intersection $K_1\cap K_2$ is trivial.  Therefore, $K_1K_2\cong K_1\times K_2$.  Since each $K_i$ is normal in $A(\Lambda)$ and since $A(\Lambda)$ has no center, either at least one of $K_1$ and $K_2$ is trivial or we can realize $F_2\times F_2$ as a subgroup of $A(\Lambda)$.  In the latter case, $\Lambda$ must contain a induced  square by Corollary~\ref{c:kambites}.
\end{proof}

The conclusion of the previous result holds in particular whenever $\Lambda$ does not split as a nontrivial join.

\begin{cor}\label{c:37}
Suppose $\Lambda$ is a square--free graph.
If $A(\Lambda)$ is centerless and contained in $A(J_1*J_2)$, then $A(\Lambda)$ embeds in $A(J_1)$ or $A(J_2)$.
\end{cor}

\begin{cor}\label{c:z star z2}
Suppose we have an embedding $\Z*\Z^2\to A(\gam)$.  Then $\gam$ contains a disjoint union of an edge and a point as an induced subgraph.
\end{cor}
\begin{proof}
Clearly we may assume that $\gam$ is $P_4$--free.
If $\gam$ is connected, then $\gam$ is a nontrivial join $\gam=J_1*J_2$. In this case, $\Z*\Z^2$ embeds in $A(J_1)$ by Corollary \ref{c:37}. The conclusion follows from induction.
Now assume $\gam$ is disconnected.  Since $\Z^2$ has rank two, at least one of the components of $\gam$ has an edge. Therefore, $\gam$ contains a disjoint union of an edge and a vertex as an induced subgraph.
\end{proof}

\begin{proof}[Proof of Corollary~\ref{c:forest}]
Let $d$ be the diameter of a largest component of $\Lambda$.
If $d\ge3$, then $A(P_4)\le \aga$ and so, Theorem \ref{t:p4embed} and Proposition~\ref{p:every forest} imply that $\gex$ contains every finite forest as an induced subgraph.

Assume $d=2$. In this case, $P_3\le\Lambda$. 
We see $P_3\le\gam$ for otherwise, $\gam$ would be complete.
So, $\gex$ contains an induced  $P_3^e$, which is the join of a vertex and an infinite discrete graph.
Hence, $\gex$ contains each component of $\Lambda$ as an induced subgraph.
If $\Lambda$ is connected or $\gam$ does not split as a nontrivial join, this would imply that $\Lambda\le\gex$; see Lemma~\ref{l:disjoint}.
Suppose $\Lambda$ is disconnected and $\gam=\gam_1\ast\gam_2$. 
Since $A(\Lambda)$ has no center, Theorem~\ref{t:join} implies that $A(\Lambda)\le A(\gam_i)$ for $i=1$ or $2$. By induction, we deduce that $\Lambda\le\gex$.
When $d\le 1$, $A(\Lambda)$ is abelian and the proof is easy.
\end{proof}

\section{Complete bipartite graphs}\label{sec:complete bipartite}
We denote the complete graph on $n$ vertices by $K_n$. The complete bipartite graph which is the join of $m$ and $n$ vertices is written as $K_{m,n}$. For convention, we also regard discrete graphs $K_{n,0}$ and $K_{0,n}$ as complete bipartite graphs.
Conjecture~\ref{c:main} is true when $\Gamma$ is complete bipartite; more precisely, one can classify right-angled Artin groups embedded in $A(K_{m,n})$ as follows.

\begin{cor}\label{c:complete bipartite}
Let $\Lambda$ be a finite graph.
\begin{enumerate}
\item
Suppose $m,n\ge2$. Then $A(\Lambda)\le A(K_{m,n})$ if and only if $\Lambda\cong K_{p,q}$ for some $p,q\ge0$.
\item
Suppose $n\ge2$. Then $A(\Lambda)\le A(K_{1,n})$ if and only if $\Lambda\cong K_{p,q}$ for some $0\le p\le 1$ and $q\ge0$.
\item
$A(\Lambda)\le A(K_{1,1})$ if and only if $\Lambda\cong K_{p,q}$ for some $0\le p,q\le 1$.
\end{enumerate}
\end{cor}

\begin{proof}
We first show that if $\gam=K_{m,n}$ and $A(\Lambda)\le A(\gam)$, then $\Lambda$ is complete bipartite.
Note that a triangle--free graph which is a non-trivial join is complete bipartite.
We have that $\Gamma$ and $\Lambda$ are triangle--free; see Lemma~\ref{l:rank}.
Recall that $P_m$ denotes a path on $m$ vertices.
Since $\gam$ does not have an induced subgraph isomorphic to $P_1\coprod P_2$, neither does $\Lambda$ by Corollary~\ref{c:z star z2}. 
So if $\Lambda$ is disconnected, then each connected component is a vertex and in particular, $\Lambda$ is discrete. 
Now we assume $\Lambda$ is connected. 
As $\Lambda$ does not contain an induced $P_4$, Lemma~\ref{l:cograph} implies that $\Lambda=J_1\ast J_2$ for some nonempty graphs $J_1$ and $J_2$ in $\mathcal{K}$. Since $\Lambda$ is triangle--free, $\Lambda$ is complete bipartite.

To complete the proof of (1), it remains to show that $A(K_{p,q})\le A(K_{m,n})$ for any $p,q\ge0$, which is clear since $C_4^e\le K_{m,n}^e$ and $C_4^e$ is the complete bipartite graph on two countable sets.

In (2), if $A(\Lambda)\le A(K_{1,n})$, then $\Lambda$ is a complete bipartite graph not containing an induced square by Corollary~\ref{c:kambites}.
The converse follows from $K_{1,2}^e = K_{1,\infty}$.
Proof of (3) is clear.
\end{proof}

\section{Co-contraction}\label{sec:cocontraction}
Let us recall the definition of the operations \emph{contraction} and \emph{co--contraction} on a graph~\cite{Kim2008}.
Let $\gam$ be a finite graph and $B$ a connected subset of the vertices; this means, the induced subgraph on $B$ is connected.
The contraction $CO(\gam,B)$ is a graph whose vertices are those of $\gam\smallsetminus B$ together with one extra vertex $v_B$, and whose edges are those of $\gam\smallsetminus B$ together with an extra edge whenever the link of a vertex in $\gam\smallsetminus B$ intersects $B$ nontrivially.  In this case we draw an edge between that vertex and $v_B$. A subset is \emph{anticonnected} if it induces a connected subgraph in ${\gam\opp}$.  
Co--contraction is defined dually for anticonnected subsets of the vertices of $\gam$.  
For instance, any pair of nonadjacent vertices is anticonnected.  We have \[\overline{CO}(\gam,B)=CO({\gam\opp},B)\opp.\]  
We can explicitly define co--contraction as follows.  The vertices of $\overline{CO}(\gam,B)$ are the vertices of $\gam\smallsetminus B$, together with an extra vertex $v_B$.  The edges of $\overline{CO}(\gam,B)$ are the edges of $\gam\smallsetminus B$.  In addition, we glue in an edge between $v_B$ and a vertex of $\gam\smallsetminus B$ if $B$ is contained in the link of that vertex.

A very easy observation is the following:
\begin{lem}\label{l:cocontract}
Let $\gam$ be a finite graph and $B$ an anticonnected subset of the vertices of $\gam$.  Then there is a sequence of graphs \[\gam=\gam_0\to\gam_1\to\cdots\to\gam_p=\overline{CO}(\gam,B)\] such that $\gam_i$ is obtained from $\gam_{i-1}$ by co--contracting $\gam_{i-1}$ relative to a pair of non--adjacent vertices.
\end{lem}

We can now give another proof of the following result which appears in \cite{Kim2008}:
\begin{thm}
Let $\gam$ be a finite graph and $B$ an anticonnected subset of the vertices of $\gam$.  Then \[A(\overline{CO}(\gam,B))\le A(\gam).\]
\end{thm}
\begin{proof}
It suffices to prove the theorem in the case where $B$ is a pair of nonadjacent vertices.  Write $B=\{v_1,v_2\}$.  
Consider the induced subgraph $\Lambda$ of $\gex$ on $V(\gam)\cup\{v_2^{v_1}\}$.
In $\Lambda$, there is an edge between $v_2^{v_1}$ and another vertex $v$ of $\Gamma$ if and only if $v$ is connected to both $v_1$ and $v_2$.  Write $v_B=v_2^{v_1}$ and delete $v_1$ and $v_2$ from $\Lambda$. Note that the resulting graph is precisely $\overline{CO}(\gam,B)$, and that we obtain the conclusion of the theorem by Theorem \ref{t:thm1}.
\end{proof}

\begin{proof}[Proof of Corollary~\ref{c:chordal}]
It is clear from the definition that ${C_n\opp}$ co--contracts onto ${C_{n-1}\opp}$ for $n\ge4$.
\end{proof}
 
For the rest of this section, we use mapping class groups to recover the theory of contraction words from \cite{Kim2008}.  
We begin with an illustrative example: Consider the graph $C_6\opp$.  We think of this graph as $C_5$ with a ``split vertex".  Precisely, label the vertices of a $5$--cycle as $\{a,b,c,d,e\}$ and then split $c$ into two vertices $v$ and $w$.  We have that $v$ and $w$ are both connected to $b$ and $d$, and we add two extra edges between $v$ and $a$ and between $w$ and $e$. 
 It is easy to check that this graph is ${C_6\opp}$; see Figure~\ref{f:c6}.

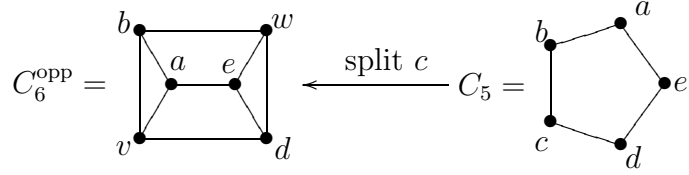
\begin{figure}[ht!] 
\[
\xymatrix{
C_6\opp=
\xy
0;/r.10pc/:
(20, 17)*{}="p";
(25, 20)*{w};
(-20, 17)*{}="a";
(-25, 20)*{b};
(-10, 0)*{}="b";
(-8, 7)*{a};
(-20,-17)*{}="c";
(-25, -20)*{v};
(20, -17)*{}="r";
(25, -20)*{d};
(10, 0)*{}="q";
(8, 6)*{e};
"p"*{\bullet};
"a"*{\bullet};
"b"*{\bullet};
"c"*{\bullet};
"r"*{\bullet};
"q"*{\bullet};
"p";"q"**\dir{-};
"p";"r"**\dir{-};
"q";"r"**\dir{-};
"a";"b"**\dir{-};
"a";"c"**\dir{-};
"b";"c"**\dir{-};
"a";"p"**\dir{-};
"b";"q"**\dir{-};
"c";"r"**\dir{-};
\endxy
& &
C_5
=
\xy
0;/r.10pc/:
(22, 20)*{}="a"; 
(29, 25)*{a};
(0, 13)*{}="b"; 
(-3,18)*{b};
(0, -11)*{}="c"; 
(-3,-18)*{c};
(22, -18)*{}="d"; 
(26,-22)*{d};
(36, 1)*{}="e";
(41, 1)*{e};
"a"*{\bullet};
"b"*{\bullet};
"c"*{\bullet};
"d"*{\bullet};
"e"*{\bullet};
"a";"b"**\dir{-};
"b";"c"**\dir{-};
"c";"d"**\dir{-};
"d";"e"**\dir{-};
"e";"a"**\dir{-};
\endxy
\ar[ll]_>>>>>>>>>>>>>>>*{\mbox{split }c\qquad\quad}
}
\]
\caption{$C_6\opp$ and $C_5$.
\label{f:c6}}
\end{figure}

\begin{prop}\label{p:c6}
If the generators of $A(C_6\opp)$ are labeled as in Figure~\ref{f:c6}, then there exists an $N$ such that for all $n\geq N$, \[\langle a^n,b^n,(vw)^n, d^n, e^n\rangle\cong A(C_5).\]
\end{prop}
\begin{proof}
Represent the vertices of ${C_6\opp}$ as simple closed curves on a surface $\Sigma$ with the correct co--incidence correspondence.  We may arrange so that $v$ and $w$ together fill a torus $T$ with one boundary component, as the curves $x$ and $y$.  Writing $v=T_x$ and $w=T_y^{-1}$, we may assume that $vw$ is a pseudo-Anosov homeomorphism supported on $T$.  By \cite{Koberda2011}, we have the conclusion since the co--incidence graph of $\{a,b,T,d,e\}$ is precisely $C_5$.
\end{proof}

In \cite{Kim2008}, the first author constructs so--called \emph{contraction words} and \emph{contraction sequences}.  
We will not give precise definitions for these terms other than if $v$ and $w$ are nonadjacent vertices in a graph $\gam$, then any word in \[\langle a,b\rangle\smallsetminus \{a^mb^n:m,n\in\bZ\}^{\pm1}\] is a contraction word.  The primary result concerning contraction words is the following:

\begin{thm}[\cite{Kim2008}]
Let $\gam$ be a graph and let $\{B_1,\ldots,B_m\}$ be disjoint, anticonnected subsets of $\gam$.  For each $i$, we write $v_{B_i}$ for the vertex corresponding to $B_i$ in $\overline{CO}(\gam,(B_1,\ldots,B_m))$, and we write $g_i$ for some contraction word of $B_i$.  Then there is an injective map \[\iota:A(\overline{CO}(\gam,(B_1,\ldots,B_m)))\to A(\gam)\] such that for a vertex $x$, we have $\iota(x)=g_i$ if $x=v_{B_i}$ and $\iota(x)=x$ otherwise.
\end{thm} 

We can partially understand this result using mapping class groups as follows:  our contraction words will be built out of two nonadjacent vertices $a$ and $b$ in $\gam$ and will be of the form $(a^nb^m)^{\pm N}$ for some $N$ sufficiently large.  To see this, we simply arrange $a$ and $b$ to correspond to two simple curves $x$ and $y$ which fill a torus $T$ with one boundary component in a large surface $\Sigma$.  Taking a power of a positive twist about $x$ and a power of a negative twist about $y$ and declaring these to be $a$ and $b$ respectively, shows that $(a^nb^m)^{\pm1}$ is always a pseudo-Anosov homeomorphism supported on $T$.  Passing to a power of this homeomorphism, we obtain that the subgroup of $A(\gam)$ generated by $(a^nb^m)^{\pm N}$ and sufficiently large powers of the other vertices of $\gam$ will isomorphic to $A(\overline{CO}(\gam,\{a,b\}))$.

\section{Triangle--free graphs and long cycles}\label{sec:long cycle}
In this section, we prove Theorem \ref{t:triangle--free}.  

\begin{proof}[Proof of Theorem \ref{t:triangle--free}] 
We suppose $A(\Lambda)$ embeds into $A(\gam)$ for some nonempty graphs $\Lambda$ and $\gam$.
We can assume that $\gam$ is not complete and does not split as a non-trivial join; 
otherwise, $\gam$ is complete bipartite and the proof is obvious from Corollary~\ref{c:complete bipartite}. By Lemma~\ref{l:disjoint}, we have only to consider the case when $\Lambda$ is connected.

Lemmas~\ref{l:cmfree} and~\ref{l:rank} imply that $\gex$ and $\Lambda$ are both triangle--free.
By Theorem~\ref{t:thm2}, there is an embedding $\phi: \Lambda \to \gex_k$ whose image is an induced subgraph.
We can further require that for any distinct vertices $v$ and $v'$ of $\Lambda$, the clique corresponding to $i(v)$ is not contained in the clique corresponding to $i(v')$.
There is a natural embedding $\psi:\gex\to \gex_k$.
If $\phi(\Lambda)$ is not contained in $\psi(\gex)$,
then for some vertex $u$ of $\Lambda$, $\phi(u) = v_{a,b}$ where $v_{a,b}\in V(\gex_k)$ corresponds to an edge $\{a,b\}$ of $\gex$. 
This implies that $\psi(a),\psi(b)\not\in \phi(V(\Lambda)\smallsetminus\{u\})$.
Since $\Lambda$ is connected and the two vertices $a$ and $b$ separate $v_{a,b}$ from the rest of $\gex_k$,
$\Lambda$ is a single vertex $\{u\}$. In particular, $\Lambda\le\gex$.
\end{proof}

We note another consequence of Theorem~\ref{t:triangle--free}, related to the Weakly Chordal Conjecture (Conjecture~\ref{c:wchordal}).

\begin{cor}\label{c:bipartite}
Let $\gam$ be a finite graph and $n\ge5$.
\begin{enumerate}
\item
Suppose $\gam$ is triangle--free.
If $A(C_n)\le\aga$ for some $n\ge5$, then $C_m\le\gam$ for some $5\le m\le n$.
\item
Suppose $\gam$ is bipartite.
If $\Lambda$ is a finite graph and $A(\Lambda)\le\aga$, then $\Lambda$ is bipartite.
\end{enumerate}
\end{cor}

\begin{proof}[Proof of Corollary~\ref{c:bipartite}]
In (1), if $\gam$ does not contain an induced $C_m$ for any $5\le m\le n$,
then Lemma~\ref{l:cmfree} would imply that $\gex$ has no induced $C_n$.
Proof of (2) is immediate from Thoerem~\ref{t:triangle--free} and by observing that the extension graph of a bipartite graph is bipartite.
\end{proof}

\begin{cor}\label{c:tsfree}
The Weakly--Chordal conjecture holds whenever $\gam$ is triangle--free or square--free.
\end{cor}
\begin{proof}
The triangle--free case was shown in Corollary~\ref{c:bipartite} (1).

A graph $\gam$ is called \emph{chordal} if it contain no induced cycle of length $n\geq 4$.  Right-angled Artin groups on chordal graphs do not contain fundamental groups of closed hyperbolic surfaces; see \cite{Kim2010} and \cite{CSS2008}. However, it is well--known that right-angled Artin groups on long cycles contain hyperbolic surface groups  \cite{SDS1989}.
Hence, right-angled Artin groups on chordal graphs do not contain $A(C_n)$ for any $n\ge5$.
This completes the square--free case.
\end{proof}

\begin{rem}\label{r:p6opp}
The smallest example of a weakly chordal graph for which Weakly Chordal Conjecture is unresolved is $P_6\opp$~\cite{CSS2008}.
It is known that $A(P_6\opp)$ does not contain $A(C_n)$ for an odd $n\ge5$~\cite{CSS2008,GS1999}.
\end{rem}

Theorem~\ref{t:triangle--free} trivially implies the Weakly--Chordal Conjecture for the case where the target graph $\gam$ is a cycle. A more precise statement on when there is an embedding from $A(C_m)$ to $A(C_n)$ is given by Thoerem~\ref{t:cntarget}.

\begin{proof}[Proof of Theorem \ref{t:cntarget}]
Let us fix one conjugate of $C_n$ in $C_n^e$ and denote it by $\Omega$.
We may assume $m,n\ge5$ by Corollary~\ref{c:complete bipartite} and Corollary~\ref{c:kambites}. By Theorem~\ref{t:triangle--free}, it suffices to prove that $C_m$ embeds in $C_n^e$ as an induced subgraph if and only if $m=n+k(n-4)$ for some $k\ge0$.

We first prove the forward implication by an induction on $m$. If $m\leq n$, the claim is trivial by Lemma~\ref{l:cmfree}.  Suppose that $\gamma\cong C_m$ is an induced subgraph of $C_n^e$, with $m>n$. Notice that $\gamma$ is not contained in one conjugate of $\Omega$ inside of $C_n^e$.  
Therefore, there exist two vertices $x,y$ in $\gamma$ such that $x$ and $y$ belong to distinct conjugates of $\Omega$ in $C_n^e$.
By Lemma~\ref{l:extension} (6), there exists a vertex $v\in C_n^e$ such that $\gamma\cap\st(v)$ contains at least two vertices, none of which are equal to $v$, and such that $\gamma\smallsetminus\gamma\cap\st(v)$ is disconnected.  Consider the graph spanned by $\gamma$ and $v$.  The vertex $v$ induces at least two more edges, but possibly more.  Taken together, these edges cellulate $\gamma$, dividing it into smaller induced cycles $\{A_1,\ldots,A_k\}$.  Any two of these cycles meet in either a path of length two, a single edge or precisely at the vertex $v$. See Figure~\ref{f:cntarget0}.

\begin{figure}[htb!]
  \tikzstyle {a}=[red,postaction=decorate,decoration={%
    markings,%
    mark=at position 1 with {\arrow[red]{stealth};}}]
  \tikzstyle {ab}=[black,postaction=decorate,decoration={%
    markings,%
    mark=at position .85 with {\arrow[black]{stealth};}}]
  \tikzstyle {ab2}=[black,postaction=decorate,decoration={%
    markings,%
    mark=at position .75 with {\arrow[black]{stealth};}}]
  \tikzstyle {b}=[blue,postaction=decorate,decoration={%
    markings,%
    mark=at position .85 with {\arrow[blue]{stealth};},%
    mark=at position 1 with {\arrow[blue]{stealth};}}]
  \tikzstyle {c}=[orange,postaction=decorate,decoration={%
    markings,%
    mark=at position .7 with {\arrow[orange]{stealth};},%
    mark=at position .85 with {\arrow[orange]{stealth};},
    mark=at position 1 with {\arrow[orange]{stealth};}
}]
  \tikzstyle {v}=[draw,shape=circle,fill=black,inner sep=0pt]
  \tikzstyle {bv}=[black,draw,shape=circle,fill=black,inner sep=1pt]
  \tikzstyle{every edge}=[-,draw]
\subfloat[(a)]{
\begin{tikzpicture}[thick]
\draw [blue,out=90,in=90] (0,2) edge (2,2)  -- (0,0);
\draw [blue,out=-90,in=-90] (2,0) edge (0,0) -- (2,2);
\draw (.8,2) node [below] {$A_1$};
\draw (.8,.1) node [] {$A_2$};
\draw (1,1) node[bv] (v) {} +(0,-.1) node[below]  {$v$};
\draw (0,.3) node[bv] (lv) {} -- (v) -- (2,.3) node[bv] (rv) {};
\draw (1,2.58) node[bv] {} +(0,.1) node[above]  {$x$};
\draw (1,-.58) node[bv] {} +(0,-.1)node[below] {$y$};
	\end{tikzpicture}
}
\hspace{.3in}
    \subfloat[(b)]{
\begin{tikzpicture}[thick]
\draw [blue,out=90,in=90] (0,2) edge (2,2)  -- (0,0);
\draw [blue,out=-90,in=-90] (2,0) edge (0,0) -- (2,2);
\draw (.8,2) node [below] {$A_1$};
\draw (.8,.1) node [] {$A_3$};
\draw (.35,1) node[] {$A_2$};
\draw (1.65,1) node[] {$A_4$};
\draw (1,1) node[bv] (v) {} +(0,-.1) node[below]  {$v$};
\draw (0,.3) node[bv] (lv) {} -- (v) -- (2,.3) node[bv] (rv) {};
\draw (1,2.58) node[bv] {} +(0,.1) node[above]  {$x$};
\draw (1,-.58) node[bv] {} +(0,-.1)node[below] {$y$};
\draw (0,1.7) node[bv] (luv) {} -- (v) -- (2,2) node[bv] (ruv) {};
	\end{tikzpicture}   
}%
\hspace{.3in}
\subfloat[(c)]{
\begin{tikzpicture}[thick]
\draw [blue,out=90,in=90] (0,2) edge (2,2)  -- (0,0);
\draw [blue,out=-90,in=-90] (2,0) edge (0,0) -- (2,2);
\draw (1,2) node [] {$A_1$};
\draw (.6,1.1) node [] {$A_2$};
\draw (1,0) node [] {$A_3$};
\draw (2,1) node[bv] (v) {} +(0,.1) node[right]  {$v$};
\draw (0,.3) node[bv] (lv) {} -- (v) -- (0,1.9) node[bv] (rv) {};
\draw (2,1.7) node[bv] {} -- (v) -- (2,.3) node[bv] (rv) {};
\draw (1,2.58) node[bv] {} +(0,.1) node[above]  {$x$};
\draw (1,-.58) node[bv] {} +(0,-.1)node[below] {$y$};
	\end{tikzpicture}
}
  \caption{Separation of a cycle, in the proof of Theorem \ref{t:cntarget}.}
  \label{f:cntarget0}
\end{figure}
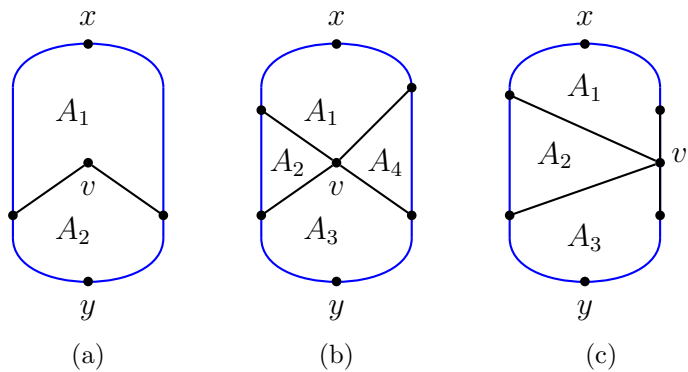

We claim that $v$ and $\gamma$ determine a cellulation of $\gamma$ which consists of exactly two induced subcycles.  To prove this, it suffices to see that there cannot be three or more induced cycles meeting at $v$ such that any two of the cycles intersect in an edge or a single vertex; that is, (b) or (c) in Figure~\ref{f:cntarget0} does not occur. Suppose we are given $k\geq 3$ such cycles ``packed" about a vertex $v$ of $C_n^e$.  First, notice that we may assume these cycles all have length $n$.  Indeed, if any one of them is longer then we can cellulate it by cycles of strictly shorter length by finding a vertex whose star separates the cycles, as above.  Now suppose that $k$ cycles of length $n$ are packed around a vertex $v$.  By Lemma~\ref{l:conjugate}, any $n$--cycles in $C_n^e$ is a conjugate of $\Omega$.  So, we may assume $v$ is a vertex in $\Omega$ with neighbors $a$ and $b$.  Since the cycles are packed about $v$, they are all conjugate to $\Omega$ by an element of the stabilizer of $v$, which is the group generated by $\{a,b,v\}$.  Since $v$ is central in this group, we may ignore it when we consider conjugates.  
As in Figure~\ref{f:cntarget} (a),
there exist $1=w_1,w_2,\ldots,w_k\in \form{a,b}\subseteq A(\gam)$ such that the following cycles are cyclically packed about $v$:
\[\{\Omega^{w_1},\ldots,\Omega^{w_k}\}.\]  
Notice that for each $i$, the word $w_i$ is a word in $a$ and $b$.  Furthermore, $w_{i-1}w_i^{-1}$ is a nonzero power of $a$ or of $b$, depending on the parity of $i$. If $k\geq 3$, $w_k$ cannot be a multiple of $a$ or $b$; however, since $\Omega^{w_k}$ and $\Omega^{w_1}$ share an edge, this would have to be the case.  So, $k=2$.

It follows that $\gamma$ is given by concatenating two induced cycles of length $k$ and $k'$ along a path of length two, so that $m=k+k'-4$.  When $n\geq 5$, square--freeness implies that $k$ and $k'$ are both smaller than $m$, which completes the induction.

Conversely, suppose that $m=n+k(n-4)$ for some $k\ge1$.  We can easily produce a copy of $C_m$ in $C_n^e$ as a ``linear" cellulation of a disk, as follows. We think of $C_n$ as the boundary of a disk.  On the boundary of each disk, choose two edge-disjoint induced paths $(a,b,c)$ and $(x,y,z)$. If $n=5$, we let $a=x$; otherwise, we assume $\{a,b,c\}$ and $\{x,y,z\}$ are disjoint. Arrange $k$ of these disks in a row and glue them together, identifying the copy of $\{a,b,c\}$ in one disk with the copy of $\{x,y,z\}$ in the next, gluing $a$ to $x$, $b$ to $y$ and $z$ to $c$. See Figure~\ref{f:cntarget} (b). The boundary of the resulting disk is clearly an induced subgraph of $C_n^e$ and has the desired length $n+k(n-4)$.
\end{proof}
\begin{figure}[htb!]
  \tikzstyle {a}=[red,postaction=decorate,decoration={%
    markings,%
    mark=at position 1 with {\arrow[red]{stealth};}}]
  \tikzstyle {ab}=[black,postaction=decorate,decoration={%
    markings,%
    mark=at position .85 with {\arrow[black]{stealth};}}]
  \tikzstyle {ab2}=[black,postaction=decorate,decoration={%
    markings,%
    mark=at position .75 with {\arrow[black]{stealth};}}]
  \tikzstyle {b}=[blue,postaction=decorate,decoration={%
    markings,%
    mark=at position .85 with {\arrow[blue]{stealth};},%
    mark=at position 1 with {\arrow[blue]{stealth};}}]
  \tikzstyle {c}=[orange,postaction=decorate,decoration={%
    markings,%
    mark=at position .7 with {\arrow[orange]{stealth};},%
    mark=at position .85 with {\arrow[orange]{stealth};},
    mark=at position 1 with {\arrow[orange]{stealth};}
}]
  \tikzstyle {v}=[draw,shape=circle,fill=black,inner sep=0pt]
  \tikzstyle {bv}=[black,draw,shape=circle,fill=black,inner sep=1pt]
  \tikzstyle{every edge}=[-,draw]
    \subfloat[(a)]{\begin{tikzpicture}[thick]
	\foreach \i in {0,...,4} {
		\draw (360/4*\i+90:2) node [bv] (v\i) {} ;
		\draw (0,0) -- (v\i);
		\draw (360/4*\i+90:2) node [bv] (w\i) {} ;
		}
		\draw (-.3,0) node [above] {$v$};
		\draw (0,0) node [bv]{};
		\draw (w1)+(-.1,0) node [left] {$b$};
		\draw (w3)+(.1,0) node [right] {$b^{w_2}$};
		\draw (w0)+(0,.1) node [above] {$a$};
		\draw (w2)+(0,-.1) node [below] {$a^{w_3}$};
		\draw [out=-120,in=-150] (w1) edge  node[pos=.5,right] (Q3) {} (w2);
		\draw [out=-120+90,in=-150+90] (w2) edge  node[pos=.5,right] (Q4) {} (w3);					
		\draw [out=-120+180,in=-150+180] (w3) edge  node[pos=.5,right]  (Q1) {}  (w0);							\draw [out=-120+270,in=-150+270] (w0) edge  node[pos=.5,right] (Q2) {}  (w1);	
		\draw (Q1)+(-.5,-.5) node[] {$\Omega^{w_2}$};	
		\draw (Q2)+(.5,-.5) node[] {$\Omega=\Omega^{w_1}$};	
		\draw (Q3)+(.5,.5) node[] {$\Omega^{w_4}$};	
		\draw (Q4)+(-.5,.5) node[] {$\Omega^{w_3}$};	
	\node  [inner sep=0.9pt] at (-90:1.5) {};  
\end{tikzpicture}}%
\hspace{.2in}
\subfloat[(b)]{
\begin{tikzpicture}[thick]
\foreach \i in {0,...,4}
    \foreach \j in {0,1,2}
		\draw (1.5*\i, \j ) node [bv] (v\i\j) {};
\foreach \i in {0,...,4}
{
\draw (v\i0)--(v\i2);
}
\foreach \i in {0,1,3}
{
\draw (v\i2)+(0,-.2) node[right] {$a$};
\draw (v\i1) node[right] {$b$};
\draw (v\i0)+(0,.2) node[right] {$c$};
}

\foreach \i in {1,2,4}
{
\draw (v\i2)+(0,-.2) node[left] {$x$};
\draw (v\i1) node[left] {$y$};
\draw (v\i0)+(0,.2) node[left] {$z$};
}

\draw [out=40,in=140] (0,2) edge (1.5,2);
\draw [out=40,in=140] (1.5,2) edge (3,2);
\draw [out=40,in=140] (4.5,2) edge (6,2);
\draw [out=-40,in=-140] (0,0) edge (1.5,0);
\draw [out=-40,in=-140](1.5,0) edge (3,0);
\draw [out=-40,in=-140](4.5,0) edge (6,0);
\draw (3.7,1) node[] {$\cdots$};
\node  [inner sep=0.9pt] at (-90:1.5) {};  
	\end{tikzpicture}
}
  \caption{Proof of Theorem \ref{t:cntarget}. (a) Cycles packed around a vertex. Note that $w_2\in\form{a}w_1=\form{a}, w_3\in\form{b}w_2, w_4\in\form{a}w_3$, and so forth. (b) A linear cellulation of a disk in $C_n^e$.}
  \label{f:cntarget}
\end{figure}
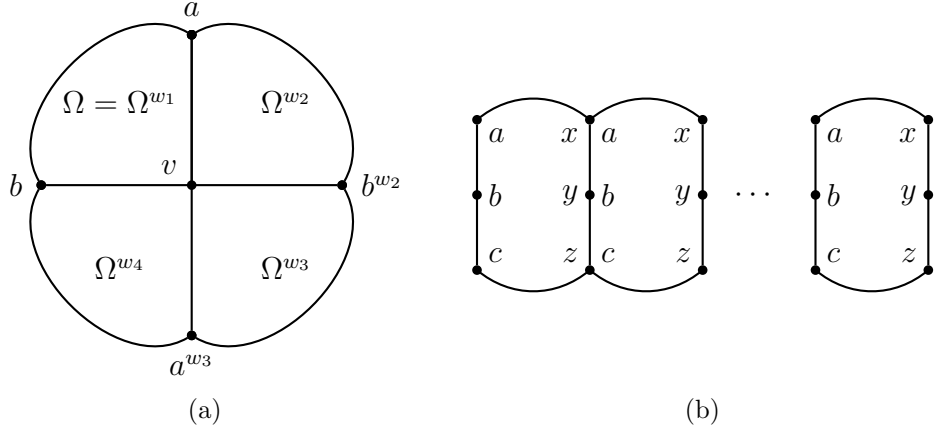

\section{Universal right-angled Artin groups}\label{s:universal}
In this short section, we prove Theorem \ref{t:universal}.
\begin{proof}[Proof of Theorem \ref{t:universal}]
If $\gam$ is triangle--free and has chromatic number $n$ then so does $\gex$; see Lemma~\ref{l:extension} (8).  Furthermore, all the induced subgraphs of $\gex$ also have chromatic number at most $n$.  It is a standard result of Erd\"os that there exist triangle--free graphs with arbitrarily large chromatic number (see \cite{Diestel}, for example), so there is no chance that $\gex$ contains every triangle--free graph.
\end{proof}

\section{Appendix: Topological Proof of Corollary~\ref{c:double}}\label{sec:appendix}
For a group $G$ and a subgroup $H$, let us denote by $G\ast_H G$ the free product of two copies of $G$ glued along the images of $H$. Also, the HNN-extension of $G$ along the identity map on $H$ is denoted as $G\ast_H$.
From~\cite{LS2001}, it easily follows that $G\ast_H G$ embeds into $G\ast_H$.
We strengthen this classical result as follows.
\begin{lem}\label{l:c:double}
Let $\phi:G\to G_1$ be a group isomorphism.
We let $P = (G\ast_H G)\ast_H, Q=G\ast_H$and call the stable generators of $P$ and $Q$ by $s$ and $t$, respectively.
We denote by $\phi$ the natural group isomorphism between the two copies of $G$ in $G\ast_H G$, regarded as subgroups of $P$.
Define $f:P\to Q$ by $f(g) = g,f(\phi(g)) = g^t$ for $g\in G$ and $f(s)=t^2$.
Then $f$ embeds $P$ into $Q$ as an index--two subgroup.
\end{lem}
\begin{proof}
Let $X$ be a CW-complex and $Y$ be a subcomplex such that $\pi_1(X)=G$ and $\pi_1(Y)=H$.
Construct a complex $Z$ for $Q$ by gluing $Y\times [0,1]$ to $Y\subseteq X$ along $Y\times 0$ and $Y\times 1$.
Then we construct a complex $Z'$ for $P$ by taking two copies $Z_1, Z_2$ of $Z$,
cutting $Z_i$ along the image of $Y\times\frac12$, and gluing those cut images in $Z_1$ to the corresponding cut images in $Z_2$, so that $Z'$ is an degree--two cover of $Z$.
\end{proof}

The reason for giving this strengthening is that it provides another proof of Theorem \ref{t:thm1} which is purely combinatorial.  The proof of Theorem \ref{t:thm1} as it is given above is the ``correct" proof since it leads to the natural generalizations which require mapping class groups in their proofs.

\begin{proof}[Proof of Theorem \ref{t:thm1}]
Lemma \ref{l:c:double} shows that $A(\gam\cup_{\st(v)}\gam)$ sits as an index--two subgroup of $A(\gam)$.  The result follows immediately.
\end{proof}
\section{Acknowledgements}
The authors thank D. Calegari, T. Church, B. Farb, L. Funar, J. Huizenga, C. McMullen, V. Gadre and S. Oum for useful comments and discussions.  Some of the fundamental ideas in this paper arose from discussions of the second author with M. Clay, C. Leininger and J. Mangahas at the Hausdorff Research Institute for Mathematics in Bonn. The first author thanks J. Manning for valuable remarks on Lemma~\ref{l:extension} (7).
\bibliography{./ref}
\bibliographystyle{amsplain}

\end{document}